\documentclass[12pt,a4paper]{amsart}%[a4,leqno,11pt]{jsarticle}%{amsart}%{jsarticle}
\usepackage{amsmath,amsthm,amssymb}
\usepackage{graphicx}
\usepackage{ulem}
\graphicspath{{./figures/}}

 \usepackage{bm}
\usepackage[usenames]{color}
\usepackage{enumerate}
\usepackage{amsthm}
\theoremstyle{plain}
 \newtheorem{theorem}{Theorem}[section]

 \newtheorem{proposition}[theorem]{Proposition}
 \newtheorem{fact}[theorem]{Fact}
 \newtheorem*{fact*}{Fact}
 \newtheorem{lemma}[theorem]{Lemma}
 \newtheorem{corollary}[theorem]{Corollary}
 \theoremstyle{remark}
 \newtheorem{definition}[theorem]{Definition}
 \newtheorem{remark}[theorem]{Remark}
 
 \newtheorem{example}[theorem]{Example}
\numberwithin{equation}{section}

\newcommand{\R}{\mathbb{R}}%{\boldsymbol{R}}
\newcommand{\C}{\mathbb{C}}%{\boldsymbol{C}}
\newcommand{\E}{\mathbb{E}}%{\boldsymbol{E}}
\renewcommand{\S}{\mathbb{S}}

\renewcommand{\Re}{{\rm Re}}
\renewcommand{\Im}{{\rm Im}}
\newcommand{\Nil}{\operatorname{Nil}_3}
\newcommand{\nil}{\mathfrak{nil}_3}
\renewcommand{\L}{\mathbb{L}^3}
\usepackage{mathtools}
\mathtoolsset{showonlyrefs=true}%ref,eqrefしたものだけ番号が付く
\title[Singularities on timelike minimal surfaces in $\Nil$]{Singularities on timelike minimal surfaces in Lorentzian Heisenberg group}
\author[S.Akamine]{Shintaro Akamine}
\address[Shintaro Akamine]{
College of Bioresource Sciences,
Nihon University, 
1866 Kameino, Fujisawa, Kanagawa, 252-0880, Japan}
\email{akamine.shintaro@nihon-u.ac.jp}
\thanks{The first author was supported by JSPS KAKENHI Grant Number 23K12979.}

\author[H.Kiyohara]{Hirotaka Kiyohara}
\address[Hirotaka Kiyohara]{
Department of Mathematics Education,
Osaka Kyoiku University,
4-698-1 Asahigaoka, Kashiwara, Osaka, 582-8582, Japan}
\email{kiyohara-h51@cc.osaka-kyoiku.ac.jp}
\thanks{The second author was supported by JST SPRING, Grant Number JPMJSP2119, and by the Foundation of Research Fellows, The Mathematical Society of Japan.}

\date{\today}

 %53C43
\begin{document}
\maketitle
\begin{abstract}
Timelike minimal surfaces in the three-dimensional Lorentzian Heisenberg group are shown to be constructed from Lorentzian harmonic maps into the de-Sitter two-sphere, and they naturally admit singular points. In particular, we provide criteria for cuspidal edges, swallowtails, and cuspidal cross caps, and present several explicit examples.
\end{abstract}
%%%%%%%%%%%%%%%%%%%%%%%%%%%%%%%%%%
%%%%%%%%%%%%%%%%%%%%%%%%%%%%%%%%%%
%%%%%%%%%%%%%%%%%%%%%%%%%%%%%%%%%%
\section{Introduction}
Minimal and constant mean curvature (CMC) surfaces are fundamental objects in differential geometry and geometric analysis, and they have been studied from many viewpoints. Moreover, many attempts have been made to generalize the theory of minimal and CMC surfaces to Minkowski $3$-space $\L$. In particular, Lorentzian analogues of classical results were established for spacelike surfaces, including Calabi--Bernstein-type rigidity theorems \cite{C, CY} as well as providing representation formulas \cite{AN, Kosamu}. These works clarified fundamental global and analytic properties of spacelike surfaces in the Lorentzian setting.

Alongside the study of spacelike surfaces, timelike minimal surfaces in $\L$ have also been investigated. In particular, M.~Magid \cite{M} studied timelike surfaces, providing a representation formula with prescribed Gauss map and mean curvature. Various properties of Lorentz surfaces have also been investigated; see, for example, \cite{W} and the references therein.

In both the spacelike and timelike settings, it became clear that singularities arise naturally, and that one cannot, in general, avoid dealing with them. In this context, M.~Umehara--K.~Yamada \cite{UY} introduced the notion of {\it maxfaces} as spacelike maximal surfaces with singular points and no branch points. Subsequently, Umehara, Yamada, and their collaborators provided explicit criteria for detecting these singularities, cuspidal edges, swallowtails \cite{UY} (see also \cite{KRSUY}), and cuspidal cross caps \cite{FSUY}, from Weierstrass-type data of maximal surfaces.

Parallel developments took place for spacelike CMC surfaces in $\L$. For such surfaces, K.~Akutagawa--S.~Nishikawa \cite{AN} derived a Kenmotsu-type representation formula, which expresses a surface in terms of its Gauss map and mean curvature. This result can be regarded as a Lorentzian analogue of Kenmotsu's formula \cite{Kenmotsu} for CMC surfaces in Euclidean $3$-space $\E^3$, and it provides a concrete method for constructing spacelike CMC surfaces from prescribed data.
These representation formulas have become a basic tool in the study of CMC surfaces.
Using Kenmotsu-type representation, Y.~Umeda \cite{Umeda} introduced a framework in which spacelike CMC surfaces in $\L$ with singular points are constructed from suitably extended harmonic data. In general, constructing CMC surfaces is difficult because their Gauss maps valued in the hyperbolic plane $\mathbb{H}^2$ are not holomorphic but harmonic, that is, the data of a CMC surface is a solution to a second-order non-linear partial differential equation. However, D.~Brander \cite{B} gave representation formulas of Bj\"orling-type for spacelike CMC surfaces with singularities via an integrable system method, the so-called DPW method. D.~Brander--M. Svensson \cite{BS} also investigated timelike CMC surfaces with singularities in the same way. 

Another line of research concerns submanifolds in homogeneous Riemannian manifolds. In particular, for the three-dimensional Heisenberg group $\Nil$, C.~Figueroa \cite{F}, J.~Inoguchi \cite{I}, and B.~Daniel \cite{D} showed independently that the normal Gauss map of a nowhere vertical minimal surface is a harmonic map into $\mathbb{H}^2$, and that minimal surfaces can be reconstructed from such data via a Kenmotsu-type representation. These results reveal a duality between minimal surfaces in $\Nil$ and spacelike CMC surfaces in $\L$.

It is natural to generalize the theory in $\Nil$ to Lorentzian settings. S.~Rahmani \cite{R} classified left invariant Lorentzian metrics on $\Nil$ and identified two families of non-flat metrics, denoted by $g_+$ and $g_-$ as in \eqref{eq: lorentzian metric}. For the metric $g_+$, S.-P.~Kobayashi and the second named author \cite{KK} showed that the normal Gauss map of a nowhere vertical timelike minimal surface induces a Lorentzian harmonic map into the de-Sitter $2$-sphere $\S^2_1$, and that such surfaces can be reconstructed from this harmonic data. Moreover, for the metric $g_-$, H.~Lee \cite{L}, together with related work by D.~Brander--S.-P.~Kobayashi \cite{BK}, established that the normal Gauss map of a spacelike maximal surface is a harmonic map into the unit sphere. As a result, timelike minimal surfaces in $(\Nil, g_+)$ correspond to timelike CMC surfaces in $\L$, and spacelike maximal surfaces in $(\Nil, g_-)$ correspond to CMC surfaces in Euclidean $3$-space $\mathbb{E}^3$.

Despite these developments, the singularity theory of zero mean curvature surfaces in $\Nil$ has remained largely unexplored.
Although recently, Brander--Kobayashi \cite{BK} studied singularities on spacelike maximal surfaces with respect to the metric $g_-$ using integrable systems techniques and established explicit criteria for generic singularities, to the best of the authors' knowledge, there is no other systematic investigation of singularities on surfaces in $\Nil$. For example, there is no known example of a timelike minimal surface in $(\Nil, g_+)$ with a specific singularity such as a swallowtail or a cuspidal cross cap.

The purpose of the present paper is to develop a theory of timelike minimal surfaces with singular points in $(\Nil, g_+)$. More precisely, we establish explicit criteria for detecting typical singularities such as cuspidal edges, swallowtails, and cuspidal cross caps, expressed in terms of harmonic data (see Theorem \ref{thm: criteria}). 
Furthermore, using the duality between timelike minimal surfaces in $\Nil$ and timelike CMC surfaces in $\L$, we characterize such singularities on surfaces in $\Nil$ via the corresponding regular points on surfaces in $\mathbb{L}^3$  (see Theorem \ref{thm: criteria2}).
In Theorem \ref{thm: existence}, we also give first examples of timelike minimal surfaces in $\mathrm{Nil}$ which have swallowtails and cuspidal cross caps.

%%%%%%%%%%%%%%%%%%%%%%%%%%%%%%%%%%
%%%%%%%%%%%%%%%%%%%%%%%%%%%%%%%%%%
%%%%%%%%%%%%%%%%%%%%%%%%%%%%%%%%%%
\section{Preliminaries}
In this section, we review the theory of timelike surfaces in the Heisenberg group and a characterization via harmonic maps with a Kenmotsu-type representation.

Let $\tau$ be a non-zero constant in $\R$. The three-dimensional Heisenberg group $\Nil(\tau)$ is a standard model of nilpotent Lie groups, and is regarded as the Lie group $\R^3$ equipped with the following group multiplication$:$
\[
\begin{pmatrix}x^1\\ x^2\\ x^3\end{pmatrix} \cdot \begin{pmatrix}y^1\\ y^2\\ y^3\end{pmatrix} = \begin{pmatrix}x^1+y^1\\ x^2+y^2\\ x^3+y^3 +\tau (x^1y^2-x^2y^1)\end{pmatrix}.
\]
The Lie algebra, denoted by $\nil(\tau)$, corresponding to $\Nil(\tau)$ is $\R^3$ endowed with the bracket $[,]$ satisfying following relations.
\[
[e_1, e_2]=2\tau e_3,
\quad
[e_2, e_3]=0,
\quad
[e_3, e_1]=0,
\]
where $e_1={}^t(1,0,0)$, $e_2={}^t(0,1,0)$, and $e_3={}^t(0,0,1)$ form the standard basis of $\R^3$ and the center is obviously given by $e_3$.
The Heisenberg group $\Nil(\tau)$ admits a Riemannian metric $g_R$ that is compatible with the group structure$:$
\begin{equation}\label{eq: Riem met}
g_R = dx^1 \otimes dx^1 + dx^2 \otimes dx^2 + \eta(\tau) \otimes \eta(\tau),
\end{equation}
where $\eta(\tau) = dx^3 + \tau (x^2 dx^1-x^1 dx^2)$.
The triplet
\[
E_1:=\frac{\partial}{\partial x^1} - \tau x^2 \frac{\partial}{\partial x^3},
\quad
E_2:=\frac{\partial}{\partial x^2} + \tau x^1 \frac{\partial}{\partial x^3},
\quad
E_3:= \frac{\partial}{\partial x^3}
\] 
gives an orthonormal basis of the space consisting of left invariant vector fields. Since the tangent space at the unit element ${}^t(0,0,0)$ and the space of left invariant vector fields are linearly isomorphic, $e_k$ can be identified with $E_k$ for $k=1,2,3$, through the natural identification between $\R^3$ and its tangent space at the origin.
We would like to note that $\eta(\tau)$ is a contact form on $\Nil(\tau)$, and the corresponding Reeb vector field is $E_3$. It is known that $\Nil(1)$ is one of the Sasakian space forms.

Although any left invariant Riemannian metric on $\Nil(\tau)$ is isometric with $g_R$, Rahmani \cite{R} showed that there exist three types of left invariant Lorentzian metrics on $\Nil(\tau)$, two non-flat metrics and one flat metric. They are classified according to whether the direction of the center $E_3$ is spacelike, timelike, or lightlike. In particular, non-flat Lorentzian metrics are given as follows:
\begin{equation}\label{eq: lorentzian metric}
\begin{aligned}
g_+&= -dx^1 \otimes dx^1 + dx^2 \otimes dx^2 + \eta(\tau) \otimes \eta(\tau),\\
g_-&= dx^1 \otimes dx^1 + dx^2 \otimes dx^2- \eta(\tau) \otimes \eta(\tau).
%g_0&= 4\tau^2dx^1 \otimes dx^1 -dx^2 \otimes dx^2 + 2dx^2 \otimes \eta(\tau).
\end{aligned}
\end{equation}
The Levi-Civita connections for $g= g_R$, $g_+$, and $g_-$, denoted by $\nabla$ are computed as follows by the Koszul formula$:$
\begin{equation*}
\begin{matrix*}[l]
\nabla_{E_1}E_1 = 0                           & \nabla_{E_1}E_2 = \tau E_3                              & \nabla_{E_1}E_3 = -\epsilon_{-}\tau E_2 \\
\nabla_{E_2}E_1 = -\tau E_3                 & \nabla_{E_2}E_2 = 0                                       & \nabla_{E_2}E_3 = \epsilon_{+} \epsilon_{-} \tau E_1\\
\nabla_{E_3}E_1 = -\epsilon_{-} \tau E_2 & \nabla_{E_3}E_2 = \epsilon_{+}\epsilon_{-}\tau E_1 & \nabla_{E_3}E_3 = 0,
\end{matrix*}
\end{equation*}
where $\epsilon_{\pm}$ are signature depending on metric
\begin{equation}
\epsilon_{+} = \left\{ \begin{array}{ccc} 1 & \text{if} & g \neq g_{+} \\ -1 & \text{if} & g = g_{+}  \end{array} \right.,
\quad
\epsilon_{-} = \left\{ \begin{array}{ccc} 1 & \text{if} & g \neq g_{-} \\ -1 & \text{if} & g = g_{-}  \end{array} \right..
\end{equation}
%%%%%%%%%%%%%%%%%%%%%%%%%%%%%%%%%%
\subsection{Timelike surfaces}
Let $\mathbb{C}'$ be the set of paracomplex numbers of the form $z=x+jy$, where $x,y\in \mathbb{R}$ and $j$ is the imaginary unit satisfying $j^2=1$. For each $z=x+jy\in \mathbb{C}'$, one can define the notions of
\begin{itemize}
\item the real part $\Re{z}:=x$ and the imaginary part $\Im{z}:=y$,
\item the conjugate $\bar{z}:=x-jy$, and
\item squared modulus of $z$ as a $|z|^2:=z\bar{z}=x^2-y^2$.
\end{itemize}
It should be remarked that the relation $|z|^2<0$ may hold in general, and $|jz|^2=-|z|^2$ holds.
Moreover, the following proposition shows a basic property of paracomplex numbers.
\begin{proposition}\label{prop: inverse}
A paracomplex number $z$ has the inverse $1/z$ if and only if $z$ satisfies that $|z|^2 \neq 0$.
\end{proposition}
\begin{proof}
First, we assume that $1/z$ is represented by $1/z= u + jv$ for some $u$ and $v \in \R$. Then we obtain $|z|^2 = 1/ (u^2 - v^2) \neq 0$.
Conversely, assuming that $|z|^2 = x^2-y^2 \neq0$, $1/z$ is represented as $x/(x^2-y^2) -j y/(x^2-y^2)$.
\end{proof}

An immersion $f \colon M \to \Nil(\tau)$ from a $2$-dimensional connected manifold $M$ into the Heisenberg group $\Nil(\tau)$ endowed with a Lorentzian metric is said to be {\it timelike} if the induced metric on $M$ is Lorentzian. Since any timelike surface $f$ can be considered as a conformal immersion from a Lorentz surface, then the induced metric of a timelike surface can be written in the form $e^udzd\bar{z}$ where the function $e^u$ is called the {\it conformal factor} of the surface.

From now on, we will endow $\Nil(\tau)$ with the metric $g_+$, and denote by $\nabla$ the Levi-Civita connection of $g_+$. The derivative of a vector field $X$ along $f$ with respect to $\partial/ \partial z$ is denoted by $\nabla_{z} X$. With respect to metric $g_+$, $E_1$ represents a timelike direction and $E_2$ and $E_3$ define spacelike directions, respectively. As a first step in our discussion of timelike surfaces, we begin with the following proposition.
\begin{proposition}\label{prop: derivative}
Let $f\colon \C' \supset D \to \Nil(\tau)$ be a timelike surface defined on a connected domain $D \subset \C'$. Then, there exist para-complex valued functions $g$ and $\hat\omega$ satisfying
\begin{equation}\label{eq: derivative}
f_z = ( g^2+1) \frac{1}{\tau}\hat\omega E_1 + j (g^2-1) \frac{1}{\tau}\hat\omega E_2 + 2g \frac{1}{\tau}\hat\omega E_3.
\end{equation}
Here, $f_z=f_* \left(\partial/\partial z\right)$.
\end{proposition}
The following lemma is useful in proving this proposition.
\begin{lemma}[{\cite[lemma $2.1$]{KK}}]\label{lem: root}
If the product $xy$ of two para-complex numbers $x$ and $y \in \C'$ has a square root, then there exists unique $\epsilon \in \{\pm1, \pm j\}$ such that $\epsilon x$ and $\epsilon y$ simultaneously have a root.
\end{lemma}
\begin{proof}[Proof of Proposition \ref{prop: derivative}]
Expand $f_z$ to $f_z = \sum_{k=1}^3 \phi^kE_k$. Since $g_+(f_z, f_z) =0$ holds, we have $(\phi^3)^2 = (\phi^1 -j \phi^2) (\phi^1 + j \phi^2)$, and hence, by Lemma \ref{lem: root}, there exist $A$, $B$, and unique $\epsilon \in \{\pm1, \pm j\}$ such that
\[
A^2= \epsilon \frac{\phi^1- j \phi^2}{2},
\quad
B^2 = \epsilon \frac{\phi^1+ j \phi^2}{2}.
\]
Although the functions $A$ and $B$ are determined only up to a multiple of the para-complex unit and a sign, they can be chosen appropriately so that $\phi^3 = 2AB$. The conformality of $f$ and the fact that $f$ is an immersion imply $|A|^2 \neq 0$. Therefore, by using Proposition \ref{prop: inverse}, we can define functions $g$ and $\hat\omega$ by
\[
g:= \epsilon \frac{B}{A}, \quad \hat\omega := \epsilon \tau A^2,
\]
and these functions yield the formula \eqref{eq: derivative}.
Moreover, since we assume the conformal factor $e^u = -4 |\epsilon|^2 (|A|^2 - |\epsilon|^2 |B|^2)^2$ to be positive, it follows that $\epsilon$ must be either $j$ or $-j$.
\end{proof}

The unit normal vector field $N$ of $f$ is given in terms of $g$ by
\begin{equation}\label{eq: unit normal}
N=\frac{g+\bar{g} }{1-|g|^2} E_1+ \frac{j(g-\bar{g}) }{1-|g|^2} E_2+ \frac{1+|g|^2}{1-|g|^2} E_3
\end{equation}
Through the left translations, the vector field $N$ can be interpreted as a map
\[f^{-1}N=\frac{g+\bar{g} }{1-|g|^2} e_1+ \frac{j(g-\bar{g}) }{1-|g|^2} e_2+ \frac{1+|g|^2}{1-|g|^2} e_3\]
taking values in $\nil$. We would like to note that the function $g$ can be regarded as $f^{-1}N$ via the stereographic projection $\pi$ of $\S^2_1 \subset \nil$ from $-e_3$, that is, $g=\pi \circ f^{-1}N$ where
\[
\pi \colon \sum_{k=1}^3 X^kE_k \mapsto \frac{X^1}{1+X^3} +j \frac{X^2}{1+X^3}.
\]
We call $N$ or $g$ the {\it normal Gauss map} of a timelike surface $f$.

A timelike surface in $\Nil(\tau)$ is called {\it nowhere vertical} if $g_+(N, E_3) \neq 0$, that is, $|g|^2 \neq -1$ holds anywhere. The harmonicity of maps into the de Sitter $2$-sphere is closely related to the minimality of nowhere vertical timelike surfaces in $\Nil(\tau)$.
To show this relation, we would like to prove the following lemmas.
\begin{lemma}\label{lem: harmonic1}
The normal Gauss map $g$ of a nowhere vertical timelike minimal surface $f$ satisfies
\begin{equation}\label{eq: harmonic1}
g_{z\bar{z}} - \frac{2\bar{g}}{1+|g|^2} g_z g_{\bar{z}} = 0.
\end{equation}
\end{lemma}
\begin{proof}
First, we show the relations between $g$ and $\hat\omega$. The Gauss--Weingarten formulas are given by
\begin{align}\label{eq: Gauss Weingarten}
\nabla_{z} f_z = u_z f_z + QN,
\quad
\nabla_z f_{\bar{z}} - \nabla_{\bar{z}} f_z =0,
\quad
\nabla_z f_{\bar{z}} + \nabla_{\bar{z}} f_z =H e^u N,
\end{align}
where $e^u$, $H$, and $N$ denote the conformal factor, the mean curvature, and the unit normal vector field of $f$, respectively.
The first equation in \eqref{eq: Gauss Weingarten} defines the Hopf differential $Q dz^2$ of $f$, and hence the Gauss--Weingarten formulas \eqref{eq: Gauss Weingarten} are reduced to
\begin{equation}\label{eq: Gauss Weingarten2}
\nabla_{z} f_{\bar{z}} = \frac12 H e^u N.
\end{equation}
By using $g$ and $\hat\omega$, each terms of \eqref{eq: Gauss Weingarten2}, with respect to $E_k$ $(k=1,2,3)$, can be represented into
\begin{equation}\label{eq: GW E1}
\frac1\tau \left( ({\bar{g}}^2 +1) \overline{\hat\omega} \right)_z - \frac{2j}{\tau} (1+|g|^2) (g-\bar{g})|\hat\omega|^2
=
-\frac{2H}{\tau^2} (1-|g|^2) (g+\bar{g}) |\hat\omega|^2,
\end{equation}
\begin{equation}\label{eq: GW E2}
-\frac1\tau \left( j(\bar{g}^2-1)\overline{\hat\omega} \right)_z -\frac{2}{\tau} (1+|g|^2) (g+\bar{g}) |\hat\omega|^2
=
-\frac{2jH}{\tau^2} (1-|g|^2) (g-\bar{g}) |\hat\omega|^2,
\end{equation}
\begin{equation}
\frac{2}{\tau} \left( \bar{g}\overline{\hat\omega} \right)_z + \frac{2j}{\tau} (1+|g|^2) (1-|g|^2) |\hat\omega|^2
=
-\frac{2H}{\tau^2} (1-|g|^2) (1+|g|^2) |\hat\omega|^2.
\end{equation}
Canceling $\bar{g}_z$ from $E_1$ term \eqref{eq: GW E1} and $E_2$ term \eqref{eq: GW E2},
we obtain
\begin{equation}\label{eq: Dirac equation1}
\hat\omega_{\bar{z}} = \frac{2|\hat\omega|^2}{\tau} \left( -H (1-|g|^2) - j\tau (1+|g|^2) \right) \bar{g}.
\end{equation}
Moreover,  substituting \eqref{eq: Dirac equation1} into \eqref{eq: GW E1} or \eqref{eq: GW E2} derives
\begin{equation}\label{eq: Dirac equation2}
g_{\bar{z}} = -\frac{1}{\tau} \left( H(1-|g|^2)^2 -j \tau (1+|g|^2)^2 \right) \overline{\hat\omega}.
\end{equation}
In the case of timelike minimal surfaces, the Gauss--Weingarten formulas \eqref{eq: Dirac equation1} and \eqref{eq: Dirac equation2} are more simple as
\begin{equation}\label{eq: Dirac equation3}
\hat\omega_{\bar{z}} = -2j |\hat\omega|^2 (1+|g|^2) \bar{g},
\quad
g_{\bar{z}} = j (1+|g|^2)^2 \overline{\hat\omega}.
\end{equation}
Finally, the equation \eqref{eq: harmonic1} can be derived by calculating $g_{z \bar{z}}$ with the above relations.
\begin{align}
g_{z\bar{z}}
=&
2j (1+|g|^2)(g_z \bar{g} + g \bar{g}_z) \overline{\hat\omega} + j (1+|g|^2)^2 \left( 2j |\hat\omega|^2 (1+|g|^2) g \right)\\
=&
2j (1+|g|^2) (g_z \bar{g} + g \bar{g}_z) \frac{j g_{\bar{z}}}{(1+|g|^2)^2} + 2 (1+|g|^2)^3 g \left( \frac{-g_{\bar{z}} \bar{g}_{z}}{(1+|g|^2)^4} \right)\\
=&
\frac{2}{(1+|g|^2)} \left( \bar{g} g_z g_{\bar{z}} + g g_{\bar{z}} \bar{g}_z \right) - \frac{2}{1+|g|^2} \left( gg_{\bar{z}} \bar{g}_z \right)\\
=&
\frac{2\bar{g}}{1+|g|^2} g_z g_{\bar{z}}.
\end{align}
\end{proof}

\begin{lemma}\label{lem: harmonic2}
Let $\nu$ be a map $\nu \colon \C' \supset D \to \S^2_1 \subset \L_{+-+}$ whose third component does not take $-1$. Then $\nu$ is Lorentzian harmonic if and only if the composition $g:=\pi \circ \nu$ of $\nu$ and the stereographic projection $\pi$ from ${}^t(0,0,-1) \in \L_{+-+}$ satisfies \eqref{eq: harmonic1}.
\end{lemma}
\begin{proof}
It is known that a map from $\C'$ into $\S^2_1$ is Lorentzian harmonic if and only if its d'Alembertian is linearly dependent on itself (see \cite{DIT}), that is, $\nu_{z \bar{z}} = \rho \nu$ holds with some function $\rho$.
The map $\nu$ is given in terms of $g$ by
\[
\nu= \frac{1}{1+|g|^2} \begin{pmatrix} g+\bar{g}\\ j(g-\bar{g})\\ 1-|g|^2 \end{pmatrix},
\]
and then a straightforward computation shows that
\begin{align}
\nu_{\bar{z}}
=
\frac{g_{\bar{z}}}{(1+|g|^2)^2}
\begin{pmatrix} 1-\bar{g}^2 \\ j (1+{\bar{g}}^2) \\ -2 \bar{g} \end{pmatrix}
+
\frac{\bar{g}_{\bar{z}}}{(1+|g|^2)^2} 
\begin{pmatrix} 1-g^2 \\ -j  (1+g^2) \\ -2g \end{pmatrix},
\end{align}
\begin{align}
\nu_{z \bar{z}}
%=&
%\frac{1}{(1+|g|^2)^2} \left( g_{z \bar{z}} - \frac{2\bar{g}}{1+|g|^2}g_z g_{\bar{z}} - \frac{2g}{1+|g|^2} \bar{g}_z g_{\bar{z}}\right)
%\begin{pmatrix} 1-\bar{g}^2 \\ j (1+{\bar{g}}^2) \\ -2 \bar{g} \end{pmatrix}\\
%&+
%\frac{1}{(1+|g|^2)^2}
%\left( \bar{g}_{z \bar{z}} - \frac{2g}{1+|g|^2} \bar{g}_z \bar{g}_{\bar{z}} - \frac{2\bar{g}}{1+|g|^2} g_z \bar{g}_{\bar{z}}\right)
%\begin{pmatrix} 1-g^2 \\ -j  (1+g^2) \\ -2g \end{pmatrix}\\
%&+
%\frac{g_{\bar{z}}}{(1+|g|^2)^2} \begin{pmatrix} -2\bar{g} \bar{g}_z \\ 2j \bar{g} \bar{g}_z \\ -2\bar{g}_z \end{pmatrix}
%+
%\frac{\bar{g}_{\bar{z}}}{(1+|g|^2)^2} \begin{pmatrix} -2g g_z \\ -2j g g_z \\ -2g_z \end{pmatrix}\\
%
%=&
%\frac{1}{(1+|g|^2)^2} \left( g_{z \bar{z}} - \frac{2\bar{g}}{1+|g|^2}g_z g_{\bar{z}} \right)
%\begin{pmatrix} 1-\bar{g}^2 \\ j (1+{\bar{g}}^2) \\ -2 \bar{g} \end{pmatrix}
%+\frac{1}{(1+|g|^2)^2}
%\left( \bar{g}_{z \bar{z}} - \frac{2g}{1+|g|^2} \bar{g}_z \bar{g}_{\bar{z}} \right)
%\begin{pmatrix} 1-g^2 \\ -j  (1+g^2) \\ -2g \end{pmatrix}\\
%&+ \frac{2g_{\bar{z}} \bar{g}_{z}}{(1+|g|^2)^3} \begin{pmatrix} -g(1-\bar{g}^2) - \bar{g}(1+|g|^2) \\ -jg(1+\bar{g}^2) + j\bar{g}(1+|g|^2) \\ 2|g|^2 -(1+|g|^2) \end{pmatrix}
%+ \frac{2g_z\bar{g}_{\bar{z}}}{(1+|g|^2)^3} \begin{pmatrix} -\bar{g}(1-g^2) - g(1+|g|^2) \\ j\bar{g} (1+g^2) -j g(1+|g|^2) \\ 2|g|^2 -(1+|g|^2) %\end{pmatrix}\\
%
=&
\frac{1}{(1+|g|^2)^2} \left( g_{z \bar{z}} - \frac{2\bar{g}}{1+|g|^2}g_z g_{\bar{z}} \right)
\begin{pmatrix} 1-\bar{g}^2 \\ j (1+{\bar{g}}^2) \\ -2 \bar{g} \end{pmatrix}\\
&+\frac{1}{(1+|g|^2)^2}
\left( \bar{g}_{z \bar{z}} - \frac{2g}{1+|g|^2} \bar{g}_z \bar{g}_{\bar{z}} \right)
\begin{pmatrix} 1-g^2 \\ -j  (1+g^2) \\ -2g \end{pmatrix}\\
&-
\frac{2g_{\bar{z}} \bar{g}_{z}}{(1+|g|^2)^3} \begin{pmatrix} g+\bar{g}\\ j(g-\bar{g})\\ 1-|g|^2 \end{pmatrix}
-
\frac{2g_z\bar{g}_{\bar{z}}}{(1+|g|^2)^3} \begin{pmatrix} g+\bar{g}\\ j(g-\bar{g})\\ 1-|g|^2 \end{pmatrix}.
\end{align}
Since, at each point, the vector
\begin{align}
&\frac{1}{(1+|g|^2)^2} \left( g_{z \bar{z}} - \frac{2\bar{g}}{1+|g|^2}g_z g_{\bar{z}} \right)
\begin{pmatrix} 1-\bar{g}^2 \\ j (1+{\bar{g}}^2) \\ -2 \bar{g} \end{pmatrix}\\
&+\frac{1}{(1+|g|^2)^2}
\left( \bar{g}_{z \bar{z}} - \frac{2g}{1+|g|^2} \bar{g}_z \bar{g}_{\bar{z}} \right)
\begin{pmatrix} 1-g^2 \\ -j  (1+g^2) \\ -2g \end{pmatrix}
\end{align}
is linearly independent of $\nu$, we can see that $\nu_{z \bar{z}}$ lies on the direction of $\nu$ if and only if \eqref{eq: harmonic1} holds.
\end{proof}

\begin{theorem}[c.f. \cite{KK}]\label{thm: identify TMS}
Let $f \colon \C' \supset D \to \Nil(\tau)$ be a nowhere vertical timelike minimal surface defined on a simply connected domain $D$. Then the normal Gauss map $g$ of $f$ is Lorentzian harmonic with respect to the metric of the de Sitter $2$-sphere $\S^2_1 \subset \L_{+-+}$. Conversely, for any Lorentzian harmonic map $g \colon D \to \C' \setminus \{w \ |\  |w|^2=-1\}$ with respect to the metric of $\S^2_1 \subset \L_{+-+}$, let $\omega=\hat\omega dz$ be a $1$-form defined by $\hat\omega = -j \bar{g}_{z} / (1+|g|^2)^2$.
Then the map $f={}^t(f^1,f^2,f^3) \colon D \to \Nil(\tau)$ defined by
\begin{align}\label{eq: TMS from g}
f^1(z) =& \frac{2}{\tau}\Re \int^z (g^2+1)\omega,\\
f^2(z) =& \frac{2}{\tau}\Re \int^z j(g^2-1) \omega,\label{eq: representation formula}\\
f^3(z) =& \frac{2}{\tau}\Re \int^z \left(2g -\tau f^2(g^2+1)+\tau f^1 j(g^2-1) \right) \omega
\end{align}
is a timelike surface, possibly with singular points, and has zero mean curvature at regular points.
\end{theorem}
\begin{proof}
The first assertion follows from Lemmas \ref{lem: harmonic1} and \ref{lem: harmonic2}, and we now prove the latter. 
Let $g$ be a Lorentzian harmonic map with respect to the metric of $\S^2_1$, that is, $g$ satisfies \eqref{eq: harmonic1}. 
It is straightforward that $g$ and $\hat\omega$ satisfy the relations \eqref{eq: Dirac equation3}. 
Then direct computations show that
\begin{align}
g_+(f_z, f_{\bar{z}}) = -\frac{2 (1-|g|^2)^2 |\hat\omega|^2}{\tau^2},
\quad
g_+(f_z, f_z) = 0,
\end{align}
and hence, locally, $f$ is a conformal immersion.
%away from $\{z \in D \mid |g|^2=1 \ \text{or} \ |\hat\omega|^2=0\}$. 
Moreover, \eqref{eq: Dirac equation3} yields $\nabla_z f_{\bar{z}} = 0$, 
and comparison with \eqref{eq: Gauss Weingarten2} completes the proof.
\end{proof}

\subsection{Duality between timelike minimal surfaces in $\Nil(\tau)$ and timelike constant mean curvature surfaces in $\L$}\label{subsec: correspondence}
Theorem \ref{thm: identify TMS} implies that nowhere vertical timelike minimal surfaces can be identified with Lorentzian harmonic maps. Therefore, as in \cite{KK}, there exists a duality between nowhere vertical timelike minimal surfaces in $\Nil(\tau)$ and constant mean curvature surfaces in $\L_{+-+}$. For convenience, we would like to discuss surfaces in $\L_{-++}$ where the semi-Euclidean metric is denoted by $\langle, \rangle$, that is,
\[
\left\langle \begin{pmatrix}x^1\\ x^2\\ x^3\end{pmatrix}, \begin{pmatrix}y^1\\ y^2\\ y^3 \end{pmatrix} \right\rangle = -x^1y^1 + x^2y^2 + x^3y^3.
\]

The corresponding surfaces in $\L_{-++}$ to timelike minimal surfaces in $\Nil(\tau)$ with harmonic map $g$ are given as follows. For a Lorentzian harmonic map $g$ satisfying \eqref{eq: harmonic1}, define a map
\begin{equation}\label{eq: TCMC representation}
f_L = \frac{2}{\tau}\Re \int^z \begin{pmatrix} g^2+1\\j(g^2-1)\\2jg \end{pmatrix}\omega.
\end{equation}
Then $f_L$ is a timelike surface in $\L_{-++}$ with the unit normal vector field
\[
N_L = \frac{-1}{1+|g|^2} \begin{pmatrix} j (g-\bar{g}) \\ g+\bar{g} \\ 1-|g|^2 \end{pmatrix}.
\]
A straightforward computation yields that $f_L$ has the constant mean curvature $\tau$ with respect to $N_L$.
Furthermore, the projective Gauss map $\pi_L \circ N_L$ via the stereographic projection $\pi_L$ from ${}^t(0,0,1)$ is given by $-jg$.

Note that the timelike minimal surface $f$ in $\Nil(\tau)$ with normal Gauss map $g$ and the unit normal vector field $N$ are constructed by \eqref{eq: unit normal} and \eqref{eq: representation formula}$:$
\begin{equation}\label{eq: df}
f= \begin{pmatrix}f^1\\f^2\\f^3 \end{pmatrix}
= \frac{2}{\tau}\Re \int^z\begin{pmatrix}  g^2+1\\
  j(g^2-1) \\
  2g -\tau f^2(g^2+1)+\tau f^1 j(g^2-1) \end{pmatrix} \omega,
\end{equation}
\[
N=\frac{g+\bar{g} }{1-|g|^2} E_1+ \frac{j(g-\bar{g}) }{1-|g|^2} E_2+ \frac{1+|g|^2}{1-|g|^2} E_3.
\]
In particular, we can see that $f$ and $f_L$ share the derivative with respect to $z$, that is,
\begin{equation}\label{eq: relation}
f_z = \phi^1E_1 + \phi^2E_2 + \phi^3E_3, \quad ({f_L})_z ={}^t(\phi^1, \phi^2, j\phi^3)
\end{equation}
holds.
Moreover, $f_L$ has constant mean curvature $\tau$ and has the Gauss map $-jg$ via the stereographic projection from ${}^t(0,0,1)$. Furthermore, straightforward computations show that its first and second fundamental forms are given by
\begin{align}
\langle df_L, df_L \rangle=-\frac{4(1+|g|^2)^2 |\omega|^2}{\tau^2},
\quad
-\langle df_L, dN_L \rangle= Q_Ldz^2 
-\frac{4 (1+|g|^2)^2 |\omega|^2}{\tau}  
+ \overline{Q_L}d\bar{z}^2,
\end{align}
where $Q_Ldz^2=\langle (f_L)_{zz}, N_L \rangle dz^2$ is the Hopf differential of $f_L$. In terms of $g$ and $\hat\omega$, $Q_L$ is denoted by
\[
Q_L = -2jg_z \frac{\hat\omega}{\tau}.
\]
\begin{remark}
In this paper, to establish a correspondence between timelike minimal surfaces in $\Nil(\tau)$ and timelike CMC $\tau$ surfaces in $\L_{-++}$, we adopt the positively oriented normal $N$ for surfaces in $\Nil(\tau)$ and the negatively oriented normal $N_L$ for those in $\L_{-++}$. This choice is necessary because the sign of the mean curvature depends on the orientation of the normal vector.
\end{remark}

For timelike minimal surfaces in $\Nil(\tau)$, the term {\it vertical} refers to the condition that the vector $E_3$ is tangent to the surfaces, that is, the unit normal vector fields have no $E_3$ component. In contrast, on the corresponding surfaces in $\L$, it denotes points where the induced metric degenerates.

For the remainder of this section, we discuss the counterparts of the Hopf differential on timelike surfaces in $\Nil(\tau)$. As is well known, the Hopf differential $Qdz^2$ is holomorphic for constant mean curvature timelike surfaces in $\L$.  In \cite{AR2004}, an analogous quadratic form of the Hopf differential on surfaces in $3$-dimensional homogeneous spaces with $4$-dimensional isometry groups is introduced. Such a quadratic form is holomorphic for constant mean curvature surfaces, and is called {\it generalized Hopf differential} or {\it Abresch--Rosenberg differential}.
In $\Nil(\tau)$, the Abresch--Rosenberg differential $Q_{AR}dz^2$ for timelike surfaces is defined by
\[
Q_{AR} = \frac{H - j\tau}{2} Q - \tau^2 g_+(f_z, E_3)^2,
\]
where $Qdz^2$ is the Hopf differential of the timelike surface in $\Nil(\tau)$.
It is easy to verify that the Abresch--Rosenberg differential is independent of the choice of conformal coordinate system, and hence it can be defined entirely on surfaces. For timelike minimal surfaces, $Q_{AR}$ is expressed in terms of $g$ and $\hat\omega$ as follows$:$
\begin{align}
Q_{AR}
=&
\frac{0 - j\tau}{2} \left( 2g_z \frac{\hat\omega}{\tau} - 2j\tau \left( 2g\frac{\hat\omega}{\tau} \right)^2 \right) - \tau^2 \left( 2g\frac{\hat\omega}{\tau} \right)^2\\
=&
%0 \left( g_z \frac{\hat\omega}{\tau} -j \tau \left( 2g \frac{\hat\omega}{\tau} \right)^2 \right) 
-j g_z \hat\omega.
\end{align}
In particular, the Abresch--Rosenberg differential of a nowhere vertical timelike minimal surface in $\Nil(\tau)$ equals $\tau/2$ times the Hopf differential of the corresponding timelike constant mean curvature $\tau$ surface (Table \ref{tb: similarities}).

\begin{table}[htbp]
  \centering
  \begin{tabular}{|l|c|c|c|c|} \hline
      &induced &  3rd. comp.  & paraholomorphic  & (normal)   \\ 
    & metric &of $N_{(L)}$ & differential & Gauss map   \\ \hline
    $H=0$ in $\Nil(\tau)$ & $-\dfrac{4(1-|g|^2)^2 |\omega|^2 }{ \tau^2} $ & $\cfrac{1+|g|^2}{1-|g|^2}$ & $-j g_z \hat\omega dz^2$ & $g$  \\ \hline
    $H=\tau$ in $\L_{-++}$& $-\dfrac{4(1+|g|^2)^2 |\omega|^2} {\tau^2} $  & $-\cfrac{1-|g|^2}{1+|g|^2}$ & $-\cfrac{2jg_z \hat\omega}{\tau} dz^2$ & $-jg$  \\ \hline
  \end{tabular}
  \caption{Similarities of corresponding surfaces $f$ in $\Nil(\tau)$ and $f_L$ in $\L_{-++}$}
  \label{tb: similarities}
\end{table}
\begin{remark}
In Table \ref{tb: similarities}, the normal Gauss map of a surface in $\Nil(\tau)$ is $g$, while the stereographic projection of the Gauss map of the corresponding surface in $\L$ is $-jg$. This is a consequence of our choice to work with surfaces in $\L_{-++}$, rather than $\L_{+-+}$. If one instead takes as the corresponding surface in $\L_{+-+}$ the surface obtained by interchanging the first and second components in \eqref{eq: TCMC representation}, then the composition of its Gauss map with stereographic projection coincides with $g$.
\end{remark}

\begin{remark}
In this section, we introduced a correspondence between a timelike minimal surface $f$ in $\Nil(\tau)$ and a timelike CMC $\tau$ surface $f_L$ in $\mathbb{L}^3$ so that they share the same Gauss map in the sense of Table \ref{tb: similarities}. It is also known that these surfaces are further related, via parallel surfaces, to timelike surfaces in $\mathbb{L}^3$ with constant Gaussian curvature $4\tau^2$. See, for example, \cite{Kokubu}.
\end{remark}
%%%%%%%%%%%%%%%%%%%%%%%%%%%%%%%%%%
%%%%%%%%%%%%%%%%%%%%%%%%%%%%%%%%%%
%%%%%%%%%%%%%%%%%%%%%%%%%%%%%%%%%%
\section{Minimal surfaces in $\Nil(\tau)$ with singularities}
In this section, we focus on singularities on timelike minimal surfaces in $\Nil(\tau)$.
In the following, let us introduce a notion of timelike minimal surfaces in $\Nil(\tau)$ with singularities.

Throughout this paper, unless otherwise stated, we assume that the harmonic map $g$ satisfies $|g|^2 \neq -1$ and \eqref{eq: harmonic1}, as in the previous section.
A point $z\in D$ is called a {\it singular point} of a timelike surface $f$ if $f$ is not immersed at $z$.
Singular points of $f$ can be characterized by using the normal Gauss map $g$ of $f$ as follows.

\begin{proposition}
\label{prop:singpt}
Let $f$ be a timelike minimal surface in $\Nil(\tau)$ as in \eqref{eq: representation formula}. The following statements hold.
\begin{itemize}
\item[$(1)$] A point $z$ satisfying $|g(z)|^2\neq \pm \infty$ is a singular point of $f$ if and only if $|g(z)|^2=1$ or $|\hat{\omega}(z)|^2=0$. In particular, such a singular point $z$ satisfies $df_z=\bf{0}$ if and only if $\hat{\omega}(z)=0$.
\item[$(2)$] A point $z$ satisfying $|g(z)|^2= \pm \infty$ is a singular point of $f$ if and only if $|g^2\hat{\omega}(z)|^2=0$. In particular, such a singular point $z$ satisfies $df_z=\bf{0}$ if and only if $\hat{\omega}(z)=g\hat{\omega}(z)=g^2\hat{\omega}(z)=0$.
\end{itemize}
\end{proposition}

\begin{proof}
First we recall that the induced metric of $f$ is written as 
\begin{equation}\label{eq:fff}
f^*g_+=-\frac{4(1-|g|^2)^2|\omega|^2}{\tau^2}=-\frac{4(1-|g|^2)^2|\hat{\omega}|^2}{\tau^2}dzd\bar{z}. 
\end{equation}
Since the derivative of $f$ is written as \eqref{prop: derivative},
$df_z=\bf{0}$ if and only if
\[
(g^2+1)\hat{\omega}=0,\quad (g^2-1)\hat{\omega}=0\quad \text{and}\quad g\hat{\omega}=0\quad \text{at $z$},
\]
which are equivalent to the conditions 
\begin{equation}\label{eq:df0}
\hat{\omega}=0,\quad g\hat{\omega}=0\quad \text{and}\quad g^2\hat{\omega}=0\quad \text{at $z$}.
\end{equation}

Let us consider the case $|g(z)|^2\neq \pm \infty$. By \eqref{eq:fff}, the condition of singular points is given as $|g(z)|^2=1$ or $|\hat{\omega}(z)|^2=0$. Moreover, \eqref{eq:df0} implies that $df_z=\bf{0}$ if and only if $\hat{\omega}(z)=0$. Hence, the proof of the assertion (1) is completed.

Next we consider the case $|g(z)|^2= \pm \infty$. In this case, by \eqref{eq:fff}, the condition of singular points is reduced to $|g^2\hat{\omega}(z)|^2=0$. In fact, if we assume $|g^2\hat{\omega}(z)|^2=0$, then $|\hat{\omega}(z)|^2=|g\hat{\omega}(z)|^2=0$ holds by the relation $0=|g^2\hat{\omega}(z)|^2=|g(z)|^2|g\hat{\omega}(z)|^2=|g^2(z)|^2|\hat{\omega}(z)|^2$.
\end{proof}

In light of Proposition \ref{prop:singpt}, we shall introduce the following class of Lorentzian harmonic maps.

\begin{definition}\label{def:extendedmin}
A Lorentzian harmonic map $g$ is said to be {\it regular} if it satisfies
\begin{itemize}
\item $\hat{\omega}$ never vanishes at each point $z$ satisfying $|g(z)|^2\neq \pm\infty$, and 
\item  $|g^2\hat{\omega}|^2$ does not vanish at each point $z$ satisfying $|g(z)|^2=\pm \infty$. 
\end{itemize}
\end{definition}

By Proposition \ref{prop:singpt} and Definition \ref{def:extendedmin}, the set of singular points of a timelike minimal surface in $\Nil(\tau)$ with a regular Lorentzian harmonic map, called the {\it singular set}, can be divided into the following types.

\begin{corollary}[cf.~\cite{Akamine}]\label{cor:devision}
The singular set $\Sigma$ of a timelike minimal surface $f$ in $\Nil(\tau)$ with a regular Lorentzian harmonic map $g$ as in \eqref{eq: representation formula} can be written as follows.
\begin{equation}\label{eq:division}
\Sigma = \Sigma^g \cup \Sigma^\omega,\quad 
\Sigma^g:=\{z\in D\mid |g(z)|^2=1 \},\quad \Sigma^\omega:=\{z\in D\mid |\hat{\omega}(z)|^2=0 \}.
\end{equation}
\end{corollary}
We would like to note that, by definition, $\Sigma^g \cap \Sigma^\omega \neq \emptyset$ in general. 

\begin{remark}\label{remark:duality}
In this paper, we assume that $|g|^2 \neq -1$, that is, the corresponding surface $f$ has no vertical points. 
However, one wishes to admit such points on the surface $f$ and their corresponding surfaces $f_L$,  we need to extend the notion of harmonic maps into $\mathbb{S}^2_1$ appropriately in the sense of~\cite{Umeda}, as follows.

We call a map $g$ into $\mathbb{S}^2_1$ an {\it extended Lorentzian harmonic map} if it satisfies
\begin{itemize}
\item $\omega=\hat{\omega} dz$ can be extended to a $\mathbb{C}'$-valued 1-form of class $C^1$ across $\{z\in D\mid |g(z)|^2=-1\}$, where 
\begin{equation}\label{eq:A}
\hat{\omega}:=\frac{- j\bar{g}_z}{(1+|g|^2)^2}, \text{ and}
\end{equation}
\item $g$ satisfies the $\mathbb{S}^2_1$-valued harmonic map equation 
\[
g_{z\bar{z}}-2(1+|g|^2)\bar{g}g_z\bar{\hat{\omega}}=0,
\]
%\[
%g_{z\bar{z}}-2j\tau (1+|g|^2)\bar{g}g_z\bar{\hat{\omega}}=0.
%g_{z\bar{z}} - \frac{2\bar{g}}{1+|g|^2} g_z g_{\bar{z}} = 0.
%\]
\end{itemize}
and the conditions in Definition \ref{def:extendedmin}. Then we obtain the timelike counterpart of extended spacelike CMC surfaces in \cite{Umeda}.
In this way, both singularities of $f$ in $\Nil(\tau)$ and of $f_L$ in $\L_{-++}$ can be treated in a unified manner.
The singular set $\Sigma_L$ of $f_L$ associated with an extended Lorenzian harmonic map $g$ is as follow (see also \eqref{eq:division}).
\[
\Sigma_L={\Sigma_L^g} \cup \Sigma^\omega,\quad {\Sigma_L^g}:=\{z\in D\mid |g(z)|^2=-1 \}.
\]
\end{remark}
%%%%%%%%%%%%%%%%%%%%%%%%%%%%%%%%%%%
%%%%%%%%%%%%%%%%%%%%%%%%%%%%%%%%%%%
\section{Types of singularities}
In this section, we give some criteria of diffeomorphism types of singular points. First, we remark notations of {\it frontal} and {\it front}.
Let $D$ be a domain in $\mathbb{R}^2$. 
A smooth map $f\colon D \longrightarrow N^3$ from $D$ into a Riemannian $3$-manifold $(N^3,g_N)$ is called a {\it frontal} if there exists a unit vector field $n$ along $f$ such that $n$ is perpendicular to $df(TD)$ with respect to $g_N$. We call $n$ the {\it unit normal vector field} of $f$. Moreover, if the map $L=(f,n)$ is an immersion, $f$ is called a {\it front}. 

A point $p\in D$ where $f$ is not an immersion is called a {\it singular point} of the frontal $f$, and we call the set of singular points of $f$ the {\it singular set}.

Let $\Omega$ be the volume element of $(N^3,g_N)$. The function $\lambda=\Omega(f_x, f_y, n)$ on $(D;x,y)$ is called the {\it signed area density function} of the frontal $f$. A singular point $p$ is called {\it non-degenerate} (resp. {\it degenerate}) if $d\lambda_p \neq0$ (resp.~$d\lambda_p=0$). 

It is easy to see that the singular set coincides with the zero set of $\lambda$, and then the condition $d\lambda \neq 0$ implies that around a non-degenerate singular point, the singular set forms a regular curve $\gamma$ by the implicit function theorem. The curve $\gamma$ is called the {\it singular curve} of the frontal $f$, with the tangent direction at each point along the curve referred to as the {\it singular direction} of $f$.
Since ${\rm rank} (df_z)=1$ at a non-degenerate singular point $z$, the kernel of $df_z$ defines a smooth vector field $\eta$ along the singular curve and is called {\it null direction} of $df_z$.

Two smooth map-germs $f_k\colon (D_k,p_k)\longrightarrow (N^3,f_k(p_k))$ ($k=1,2$) are {\it $\mathcal{A}$-equivalent} if there exist diffeomorphism germs $\Phi\colon (\mathbb{R}^2,p_1) \longrightarrow (\mathbb{R}^2,p_2)$ and $\Psi\colon (N^3,f_1(p_1)) \longrightarrow (N^3,f_2(p_2))$ such that $f_2 = \Psi \circ f_1\circ \Phi^{-1}$. As typical singular points of frontals, a {\it cuspidal edge}, a {\it swallowtail} and a {\it cuspidal cross cap} are given as singularities which are $\mathcal{A}$-equivalent to 
\[
f_{\rm CE}(x,y) = \begin{pmatrix} x\\y^2\\y^3 \end{pmatrix},
\quad
f_{\rm SW}(x,y) = \begin{pmatrix} 3x^4+x^2y\\4x^3+2xy\\y \end{pmatrix},
\quad
f_{\rm CCR}(x,y) = \begin{pmatrix} x\\y^2\\xy^3 \end{pmatrix}
\] 
at origin, respectively. Cuspidal edges and swallowtails are singular points of fronts, and cuspidal cross caps are singular points of frontals but not fronts.

In this paper, we consider the situation of $N^3=\Nil(\tau)$.  
 Since the Heisenberg group $\Nil(\tau)$ is regarded as the space $\mathbb{R}^3$ as in Section 2, and one can consider the left-invariant Riemannian metric $g_R$ in \eqref{eq: Riem met}, we may regard $(N^3, g_N)$ as $(\Nil(\tau), g_R) = (\mathbb{R}^3, g_R)$. 
Since the volume element $\Omega$ of the Riemannian Heisenberg group $(\Nil(\tau), g_R)$ is $\Omega=dx^1\wedge dx^2\wedge dx^3$, the signed area density function $\lambda$ is written as 
\[
\lambda = \Omega(f_x, f_y, n) = \det{(f_x, f_y, n)}.
\]
The following Facts state well-known criterion for distinguishing types of singularities.

\begin{fact}[\cite{KRSUY}]\label{Fact: KRSUY}
Let $f\colon D \longrightarrow \mathbb{R}^3$ be a front and $p\in D$ a non-degenerate singular point of $f$. Take a singular curve $\gamma=\gamma(t)$ with $\gamma(0)=p$ and a vector field of null directions $\eta(t)$. Then
\begin{itemize}
\item[(i)] $p$ is a cuspidal edge if and only if $\det\left(\gamma'(0), \eta(0)\right)\neq 0$.
\item[(ii)] $p$ is a swallowtail if and only if $\det\left(\gamma'(0), \eta(0)\right)=0$ and $ \left. \frac{d}{dt}\det\left(\gamma'(t), \eta(t)\right)\right|_{t=0}\neq0$.
\end{itemize}
\end{fact}

\begin{fact}[\cite{FSUY}]\label{Fact: FSUY}
Let $f\colon D \longrightarrow \mathbb{R}^3$ be a frontal and $p\in D$ a non-degenerate singular point of $f$. Take a singular curve $\gamma=\gamma(t)$ with $\gamma(0)=p$ and a vector field of null directions $\eta(t)$. Then $p$ is a cuspidal cross cap if and only if \begin{center}$\det\left(\gamma'(0), \eta(0)\right)\neq0$, $\det\left(df(\gamma'(0)), n(0), dn(\eta(0))\right)=0$ and $\left.\frac{d}{dt}\det\left(df(\gamma'(t)), n(t), dn(\eta(t))\right)\right|_{t=0}\neq0$.\end{center}
\end{fact}

In the end of this subsection, we give the following proposition.

\begin{lemma}\label{lemma:lambda}
For each timelike minimal surface $f$ in $\Nil(\tau)$ with a regular Lorentzian harmonic map $g$, the following statements hold.
\begin{itemize}
\item[$(1)$] $f$ is a frontal with the unit normal vector field 
\begin{equation}\label{eq:N_R}
N_R:=\frac{1}{\sqrt{2(g+\bar{g})^2+(1-|g|^2)^2}}\{-(g+\bar{g})E_1+j(g-\bar{g})E_2+(1+|g|^2)E_3 \},
\end{equation}
with respect to the left invariant Riemannian metric $g_R$.
\item[$(2)$] The signed area density function $\lambda$ is expressed as
\begin{equation}\label{eq:lambda}
\lambda=-\frac{2}{\tau^2} |\hat\omega|^2 (1-|g|^2)\sqrt{2(g+\bar{g})^2+(1-|g|^2)^2}.
\end{equation}
\end{itemize}
\end{lemma}
\begin{proof}
The Riemannian vector product $\times_R$ is defined by
\[
g_R \left( \left( \sum_{i=1}^3 X^i E_i \right) \times_R \left( \sum_{j=1}^3 Y^j E_j \right), \sum_{k=1}^3 Z^k E_k \right)
=
\det \begin{pmatrix} X^1 & Y^1 & Z^1\\X^2 & Y^2 & Z^2\\ X^3 & Y^3 & Z^3 \end{pmatrix}.
\]
Then a direct computation yields
\[
f_x \times_R f_y = -\frac{2}{\tau^2} |\hat\omega|^2 (1-|g|^2) \left( -(g+\bar{g})E_1+j(g-\bar{g})E_2+(1+g\bar{g})E_3 \right),
\]
and hence the assertions follow.
\end{proof}

\subsection{Singularities in $\Sigma^\omega$}
By Corollary \ref{cor:devision}, the singular set $\Sigma$ of a timelike minimal surface $f$ in $\Nil(\tau)$ with a regular Lorentzian harmonic map $g$ is decomposed into $\Sigma=\Sigma^g\cup \Sigma^\omega$ as in \eqref{eq:division}. We first see singular points in $\Sigma^{\omega}$.
In this subsection, we discuss the most typical singular points in $\Sigma^\omega$. Based on Lemma \ref{lemma:lambda}, we first give a non-degenerate condition of singular points in $\Sigma^\omega$.

\begin{lemma}\label{lemma:nondeg}
A singular point $z\in \Sigma^\omega\setminus \Sigma^g$ of a timelike minimal surface in $\Nil(\tau)$ with a regular Lorentzian harmonic map $g$ is non-degenerate if and only if $\hat{\omega}_z\overline{\hat{\omega}}(z)\neq 0$. 
On the other hand, any singular point $z\in \Sigma^\omega\cap \Sigma^g$ is degenerate.
\end{lemma}

\begin{proof}
First, we show the first part of the assertion. Since $z\in  \Sigma^\omega\setminus \Sigma^g$, the relations $\hat{\omega}\overline{\hat{\omega}}=0$ and $(1-g\bar{g})\neq 0$ hold at $z$. Hence, $d\lambda_z\neq \bf{0}$ if and only if $(\hat{\omega}\overline{\hat{\omega}})_z(z)\neq 0$.

Let us recall that the following statements hold.
\begin{equation}\label{eq:omega_z}
\hat{\omega}_{\bar{z}}=-2j\bar{g}(1+\bar{g})\hat{\omega}\overline{\hat{\omega}},\quad g_{\bar{z}}=j(1+g\bar{g})^2\overline{\hat{\omega}}.
\end{equation}
Then, $z\in  \Sigma^\omega$ also implies $\hat{\omega}_{\bar{z}}(z)=0$.
Therefore, $(\hat{\omega}\overline{\hat{\omega}})_z(z)\neq 0$ is equivalent to $\hat{\omega}_z\overline{\hat{\omega}}(z)\neq 0$.

The latter assertion follows immediately from \eqref{eq:lambda}. 
\end{proof}

Lemma \ref{lemma:nondeg} enables us to determine diffeomorphism types of non-degenerate singular points in $\Sigma^\omega$ satisfying the front condition.

\begin{theorem}
The cuspidal edge is only non-degenerate singular point belonging to $\Sigma^\omega$ on which the considered timelike minimal surface in $\Nil(\tau)$ with a regular Lorentzian harmonic map is a front. 
\end{theorem}

\begin{proof}
Let $z\in \Sigma^\omega$ be a non-degenerate singular point.
Then, by Lemma \ref{lemma:nondeg}, without loss of generality we may assume that the image of the singular curve $\gamma = \gamma(t)$ with $z=\gamma(0)$ also lies on $\Sigma^\omega \setminus \Sigma^g$. Based on this fact, we find the singular direction $\gamma'(0)$ and the null direction $\eta(0)$ below. 

Taking the derivative of $\hat{\omega}\overline{\hat{\omega}}(\gamma(t))=0$, we have 
\[
\left( \hat{\omega}\overline{\hat{\omega}}\right)_z\frac{d\gamma}{dt}+\overline{\left( \hat{\omega}\overline{\hat{\omega}}\right)_z}\overline{\frac{d\gamma}{dt}}=0.
\]
By Lemma \ref{lemma:nondeg}, $\left( \hat{\omega}\overline{\hat{\omega}}\right)_z \neq 0$ holds on $\gamma$ and hence we can take the singular direction
\begin{equation}\label{eq:sing_dir}
\gamma'=\frac{d\gamma}{dt}=j\overline{\left( \hat{\omega}\overline{\hat{\omega}}\right)_z}=j\overline{\hat{\omega}_z}\hat{\omega}.
\end{equation}
On the other hand, since the differential $df$ is given by
\[
df
=
\frac{1}{\tau}\left\{ \left( (g^2+1)\omega + ({\bar{g}}^2+1)\bar{\omega} \right) E_1+j \left( (g^2-1)\omega - ({\bar{g}}^2 -1)\bar{\omega} \right)E_2+2\left( g\omega + \bar{g}\bar{\omega} \right)E_3 \right\},
\]
the null vector filed $\eta$ along $\gamma$ is $\eta(t) =\overline{\hat{\omega}}$ in this setting. Therefore, we obtain
\begin{equation}\label{eq:det}
\det{\left(\gamma', \eta\right)}= \Im{\left(\overline{\gamma'}\eta\right)}
=-\Re{\left( \hat{\omega}^2\overline{\hat{\omega}_z}\right)},
\end{equation}
and we will show that it never vanishes below.

If we assume that $\det{\left(\gamma', \eta\right)}=0$, then $\Re{\left( \hat{\omega}^2\overline{\hat{\omega}_z}\right)}=0$ by \eqref{eq:det}. 
Since $|\hat{\omega}^2\overline{\hat{\omega}_z}|^2=|\hat{\omega}|^2=0$, it means $\hat{\omega}^2\overline{\hat{\omega}_z}=0$ at $z$. By using the relation $|\hat{\omega}|^2=0$ again, $\hat{\omega}$ can be written as $\hat{\omega}=k(1\pm j)$ for some non-zero real number $k$ at $z$. Hence, $\hat{\omega}^2\overline{\hat{\omega}_z}=0$ implies $\hat{\omega}_z=l(1\pm j)$ for some non-zero real number $l$. 
By using these $\hat{\omega}$ and $\hat{\omega}_z$, we have
\[
\hat{\omega}_z\overline{\hat{\omega}}=kl(1\pm j)(1\mp j)=0,
\]
which contradicts to the non-degeneracy of $z$ in Lemma \ref{lemma:nondeg}. Therefore, we obtain $\det{\left(\gamma', \eta\right)}\neq 0$ at $z$ proving the assertion.
\end{proof}

%%%%%%%%%%%%%%%%%%%%%%%%%%%%%%%%%%
\subsection{Singularities in $\Sigma^g \setminus \Sigma^\omega$}
In this subsection, we treat non-degenerate singular points in $\Sigma^g$. Since singular points in $\Sigma^g \cap \Sigma^\omega$ are degenerate by Lemma \ref{lemma:nondeg}, we may restrict ourselves to singular points belonging to $\Sigma^g \setminus \Sigma^\omega$. We first derive a non-degeneracy condition for singular points in $\Sigma^g \setminus \Sigma^\omega$ from Lemma \ref{lemma:lambda}.
\begin{lemma}\label{lem: non-degeneracy of g-singular}
A singular point $z$ in $\Sigma^g \setminus \Sigma^\omega$ of a timelike minimal surface $f$ in $\Nil(\tau)$ with a regular Lorentzian harmonic map $g$ is non-degenerate if and only if $(|g|^2)_z \neq 0$ at $z$, or equivalently
\[
\frac{g_z}{g^2\hat\omega} \neq 4j.
\]
\end{lemma}
\begin{proof}
The signed area density function $\lambda$ of $f$ satisfies
\[
d\lambda = -\frac{2}{\tau^2} |\hat\omega|^2 \sqrt{2(g+\bar{g})^2} \left(-(|g|^2)_zdz -(|g|^2)_{\bar{z}}d\bar{z} \right)
\quad
\text{at $z$}.
\]
Since $g(z)$ lies on the hyperbola $\{\pm \cosh t + j \sinh t \in \C'|\ t \in \R \}$, we have $g+\bar{g} \neq 0$ at $z$.
Hence the non-degeneracy condition $d\lambda(z) \neq 0$ is equivalent to
\begin{equation}\label{eq: non-degeneracy of g-singular}
(|g|^2)_z \neq 0 \quad \text{at $z$}.
\end{equation}
Moreover, the assumption $|\hat\omega(z)|^2 \neq 0$ allows us to divide by $(g \bar{g}_z)(z)$. Therefore, using \eqref{eq: Dirac equation3}, the condition \eqref{eq: non-degeneracy of g-singular} can be rewritten as
\[
0\neq
\frac{g_z \bar{g} + g \bar{g}_z}{g \bar{g}_z}=
\frac{g_z}{g^2 \bar{g}_z} + 1=
\frac{g_z}{g^2 (-4j \hat\omega)} + 1.
\]
This completes the proof.
\end{proof}
Since the singular curve $\gamma$ of a timelike minimal surface $f$ with a regular Lorentzian harmonic map $g$ satisfies $|g(\gamma(t))|^2=1$ in this case, some types of non-degenerate singular points in $\Sigma^g \setminus \Sigma^\omega$ can be characterized in terms of the normal Gauss map $g$ of $f$ by applying Facts \ref{Fact: KRSUY} and \ref{Fact: FSUY}.
\begin{theorem}\label{thm: criteria}
Let $z$ be a non-degenerate singular point in $\Sigma^g \setminus \Sigma^\omega$ of a timelike minimal surface $f$ with a regular Lorentzian harmonic map $g$. Then the following assertions hold.
\begin{enumerate}
\item $f$ is a front at $z$ if and only if
\[
\Re \frac{g_z}{g^2\hat\omega} \neq 0 \quad \text{at $z$},
\]
\item $f$ is $\mathcal{A}$-equivalent to the cuspidal edge at $z$ if and only if
\[
\Re \frac{g_z}{g^2\hat\omega} \neq 0,\quad
\Im \frac{g_z}{g^2\hat\omega} \neq 4\quad
\text{at $z$},
\]
\item $f$ is $\mathcal{A}$-equivalent to the swallowtail at $z$ if and only if
\[
\Re \frac{g_z}{g^2\hat\omega} \neq 0,
\quad
\Im \frac{g_z}{g^2\hat\omega} = 4,
\quad
\Im \left[ (|g|^2)_{\bar{z}} \left( \Im \frac{g_z}{g^2\hat\omega} \right)_z \right] \neq 0 \quad \text{at $z$},
\]
\item $f$ is $\mathcal{A}$-equivalent to the cuspidal cross cap at $z$ if and only if
\[
\Re \frac{g_z}{g^2\hat\omega} = 0,
\quad
\Im \frac{g_z}{g^2\hat\omega} \neq 4,
\quad
\Im \left[ (|g|^2)_{\bar{z}} \left( \Re \frac{g_z}{g^2\hat\omega} \right)_z \right] \neq 0 \quad \text{at $z$},
\]
\end{enumerate}
\end{theorem}
\begin{proof}
At a non-degenerate singular point in $\Sigma^g \setminus \Sigma^\omega$, a straightforward calculation yields
\[
df
=f_zdz + f_{\bar{z}}d\bar{z}
=\frac{1}{\tau}(g\omega + \bar{g}\bar{\omega}) \left\{ (g+\bar{g})E_1 + j (g-\bar{g})E_2 + 2E_3 \right\}.
\]
Therefore, the null direction $\eta$ of $df$ is $\eta(t)=-j\overline{g \hat\omega}$, where each tangent plane is identified with $\C'$ as $a \frac{\partial}{\partial z} + \bar{a}\frac{\partial}{\partial \bar{z}} \cong a \in \C'$.

On the other hand, the null direction of $dN_R$ satisfying $dN_R(\mu) = 0$ is given by $\mu(t)=\left(\overline{\bar{g}g_z - g \bar{g}_z}\right)(\gamma(t))$, since a simple calculation based on \eqref{eq:N_R} yields
\begin{align}\label{eq: dNR}
dN_R
=&
%d \left[ \frac{1}{\sqrt{2(g+\bar{g})^2 + (1-|g|^2)^2}} \right] \left\{ -(g+\bar{g})E_1 +j(g-\bar{g}) E_2 +(1+|g|^2) E_3 \right\}\\
%&+
%\frac{1}{\sqrt{2(g+\bar{g})^2 + (1-|g|^2)^2}} d\left[ -(g+\bar{g})E_1 +j(g-\bar{g}) E_2 +(1+|g|^2) E_3 \right]\\
%=&
%\frac{j (\bar{g}dg - gd\bar{g})}{\sqrt {2|g+\bar{g}| (g+\bar{g})} }\left\{ 2E_2 -j(g-\bar{g}) E_3 \right\}.
\frac{j (\bar{g}dg - gd\bar{g})}{(g+\bar{g})\sqrt {2(g+\bar{g})^2} }\left\{ 2E_2 -j(g-\bar{g}) E_3 \right\}.
\end{align}
Hence, by using the second equation in \eqref{eq:omega_z},
\[
\det (\eta(t)\ \mu(t))
%= (\Re \eta)(\Im \mu)(0) - (\Im \eta)(\Re \mu)(0)
=\Im (\overline{\eta(t)} \mu(t))
=|\hat\omega|^2 \Im \left( j \overline{\frac{g_z}{g^2\hat\omega}} - 4 \right)
=|\hat\omega|^2 \Re \left( \frac{g_z}{g^2\hat\omega} \right).
\]
Thus, the assertion (1) follows.

Next, we rewrite the criterion for the cuspidal edge. 
Since the singular curve $\gamma$ of $f$ satisfies $|g(\gamma(t))|^2 = 1$ for $t$ near $0$, then one may take
\begin{equation}\label{eq: gamma'}
\gamma'(t) 
= j \left(\bar g g_{\bar{z}} + g \overline{g_z}\right)(\gamma(t))=j\left( |g|^2\right)_{\bar{z}}(\gamma(t)).
\end{equation}
Hence, based on \eqref{eq: Dirac equation3}, a straightforward computation shows
\begin{equation}\label{eq: det}
\det\left(\gamma'\ \eta\right)
= \Im \left( \overline{\gamma'} \eta \right)
= |\hat\omega|^2 \Im \left( \frac{g_z}{g^2\hat\omega} + \frac{\bar{g}_z}{\hat\omega} \right)
%= \epsilon j \Im \left( \frac{g_z}{g^2\hat\omega} + \frac{\bar{g}_z}{\hat\omega} \right)
= |\hat\omega|^2 \Im \left( \frac{g_z}{g^2\hat\omega} -4j \right).
%= \epsilon j \Im \left( \frac{g_z}{g^2\hat\omega} -4j \right).
\end{equation}
Therefore, by Fact \ref{Fact: KRSUY}, the assertion (2) follows.

Next, we prove (3). 
By Fact \ref{Fact: KRSUY}, a non-degenerate singular point $z=\gamma(0)$ 
induces a swallowtail if and only if $t=0$ is a simple zero of $\det(\gamma'(t)\,\eta(t))$.
From \eqref{eq: gamma'} and \eqref{eq: det} and by assuming that $\det\left(\gamma'\ \eta\right)=0$ at $t=0$, a direct computation yields, at $t=0$,
\begin{align}
\frac{d}{dt}\det(\gamma'\ \eta)
&=
|\hat\omega|^2
\frac{d}{dt}\left(\Im \frac{g_z}{g^2\hat\omega} \right) \\
&=
2 |\hat\omega|^2\Im \left[ (|g|^2)_{\bar{z}} \left( \Im \frac{g_z}{g^2\hat\omega} \right)_z \right].
\end{align}
This proves (3).

Finally, we prove the assertion (4). A straightforward calculation using $|g(\gamma(t))|^2=1$, \eqref{eq:N_R}, \eqref{eq: dNR}, and \eqref{eq: gamma'} yields
\begin{align}
\Psi(t) 
&=
\det\Bigl( df(\gamma'(t)), dN_R(\eta(t)), N_R(\gamma(t)) \Bigr) \\
&= \frac{4| \hat\omega|^4}{\tau} 
\left( \Im \frac{g_z}{g^2\hat\omega} -4 \right) 
\left(\Re \frac{g_z}{g^2\hat\omega} \right).
\end{align}
By Fact \ref{Fact: FSUY}, $f$ is $\mathcal{A}$-equivalent to a cuspidal cross cap at $z=\gamma(0)$ if and only if the singular direction $\gamma'(0)$ and the null direction $\eta(0)$ are transverse, and $t=0$ is a zero of $\Psi$ of order~$1$.
First, by \eqref{eq: det}, it is clear that the transversality of the singular direction $\gamma'(0)$ and the null direction $\eta(0)$ is equivalent to
\begin{equation}\label{eq: ccr2}
\Im \frac{g_z}{g^2\hat\omega} \neq 4 \quad \text{at $z$}.
\end{equation}
Assuming \eqref{eq: ccr2}, the condition $\Psi(0)=0$ is equivalent to
\[
\Re \frac{g_z}{g^2\hat\omega} = 0 \quad \text{at $z$}.
\]
 A straightforward computation then shows that, at such a point,
\begin{align}
\frac{d\Psi}{dt} 
= \frac{8 |\hat\omega|^4}{ \tau } 
\left( \Im \frac{g_z}{g^2\hat\omega} -4 \right) 
\Im \left[ (|g|^2)_{\bar{z}} \left( \Re \frac{g_z}{g^2\hat\omega} \right)_z \right].
\end{align}
This proves assertion (4).
\end{proof}

\begin{remark}
If we define the function
\[
  \hat{A} := \frac{j}{\tau} \hat{\omega},
  %  \hat{A} := \frac{\epsilon}{\tau} \hat{\omega},
\]
we can rewrite the three criteria in Theorem~\ref{thm: criteria} as follows:
$f$ is $\mathcal{A}$-equivalent to
\begin{itemize}
\item[(2')]
cuspidal edge at $z$ if and only if
\[
  \Im \frac{g_z}{g^2\hat{A}} \neq 0\quad
  \text{and}\quad
  \Re \frac{g_z}{g^2\hat{A}} \neq 4 \tau \quad
  %  \Re \frac{g_z}{g^2\hat{A}} \neq 4\epsilon j \tau \quad
  \text{at $z$},
\]
\item[(3')]
swallowtail at $z$ if and only if
\[
  \Im \frac{g_z}{g^2\hat{A}} \neq 0,\quad
  \Re \frac{g_z}{g^2\hat{A}} = 4 \tau, \quad
  %  \Re \frac{g_z}{g^2\hat{A}} = 4\epsilon j \tau, \quad
  \text{and}\quad 
  \Im \left[(|g|^2)_{\bar{z}} \left(\Re \frac{g_z}{g^2\hat{A}}\right)_z\right] \neq 0 \quad
  \text{at $z$},
\]
\item[(4')]
cuspidal cross cap at $z$ if and only if
\[
  \Im \frac{g_z}{g^2\hat{A}} = 0,\quad
  \Re \frac{g_z}{g^2\hat{A}} \neq 4 \tau, \quad
  %  \Re \frac{g_z}{g^2\hat{A}} \neq 4\epsilon j \tau, \quad
  \text{and}\quad
  \Im \left[(|g|^2)_{\bar{z}} \left(\Im \frac{g_z}{g^2\hat{A}}\right)_z\right] \neq 0 \quad
  \text{at $z$}.
\]
\end{itemize}
Using $\hat{A}$, the derivative $f_z$ can be written as
\[
  f_z
  =
  j (g^2+1)\hat{A}E_1
  +
  (g^2-1)\hat{A}E_2
  +
  2 j g\hat{A}E_3.
\]
Accordingly, the Gauss--Weingarten equations \eqref{eq: Dirac equation3} become
\[
  \hat{A}_{\bar{z}}
  =
  2 \tau |\hat{A}|^2(1+|g|^2) \bar{g},
  \quad
  g_{\bar{z}}
  =
  - \tau (1+|g|^2)^2 \overline{\hat{A}}.
\]
Therefore, formally letting $\tau=0$, we see that $g$ becomes para-holomorphic.
As a consequence, simple calculations yield
\[
  (|g|^2)_{\bar{z}} = g\bar{g}_{\bar{z}}\quad
  \text{and}\quad
  \quad
  \left(\frac{g_z}{g^2\hat{A}}\right)_{\bar{z}}
  =
  -4 \tau \bar{g} \bar{\hat{A}} \frac{g_z}{g^2\hat{A}}
  = 0.
\] 
Hence, we obtain
\begin{align}
  \Im\!\left[(|g|^2)_{\bar z}\left(\Re\frac{g_z}{g^2\hat{A}}\right)_z \right]
  &=
  \frac{|g_z|^2}{2}
  \left\{
    \Im\!\left[
      \frac{g}{g_z}
      \left(\frac{g_z}{g^2\hat{A}}\right)_z
    \right]
  \right\},\\
  \Im\!\left[(|g|^2)_{\bar z}\left(\Im\frac{g_z}{g^2\hat{A}}\right)_z \right]
  &=
  \frac{|g_z|^2}{2}
  \left\{
    \Re\!\left[
      \frac{g}{g_z}
      \left(\frac{g_z}{g^2\hat{A}}\right)_z
    \right]
  \right\}.
\end{align}
These relations formally recover the criteria for singularities of timelike minimal surfaces in Minkowski space (see \cite{Akamine, Takahashi}).
\end{remark}

Via the correspondence between a timelike minimal surface $f$ and a timelike CMC $\tau$ surface $f_L$ in $\L$ given in the subsection \ref{subsec: correspondence}, it follows that a singular point $z$ of $f$ belongs to $\Sigma^g$ if and only if the condition
\[
\langle N_L, e_3 \rangle = 0 \quad \text{at $z$}
\]
holds where $N_L$ is the unit normal vector field of $f_L$. The derivative of the third component of $N_L$ is
\[
\frac{\partial}{\partial z} \langle N_L, e_3 \rangle = \frac{\partial}{\partial z} \left( - \frac{1-|g|^2}{1+|g|^2} \right) = \frac{2(|g|^2)_z}{(1+|g|^2)^2}.
\]
Hence the following lemma is an immediate consequence of Lemma \ref{lem: non-degeneracy of g-singular}.
\begin{lemma}
Let $z$ be a singular point of a timelike minimal surface $f$ in $\Nil(\tau)$ with a regular Lorentzian harmonic map $g$ and assume that $z \in \Sigma^g$. Then $z$ is non-degenerate if and only if $\langle dN_L, e_3 \rangle \neq 0$ at $z$.
\end{lemma}

The singular curve $\gamma$ consisting of non-degenerate singular points in $\Sigma^g$ of $f$ gives rise to a regular curve in the corresponding $f_L$. Based on this fact, singularities appearing in Theorem \ref{thm: criteria} can be characterized in terms of the corresponding regular curve, rather than the normal Gauss map $g$.
\begin{theorem}\label{thm: criteria2}
Let $z$ be a non-degenerate singular point of a timelike minimal surface $f$ in $\Nil(\tau)$ with a regular Lorentzian harmonic map $g$ belonging to $\Sigma^g$. Moreover let denote $\gamma$ the singular curve of $f$ passing through $z$ at $t=0$, and put $\gamma_L:=f_L \circ \gamma$. Then, the following statements hold.
\begin{itemize}
\item[(1)]
$f$ is a front at $z$ if and only if $\langle {\gamma_L}', e_3 \rangle \neq 0$ at $z$,
\item[(2)]
$f$ is $\mathcal{A}$-equivalent to the cuspidal edge at $z$ if and only if
\[
\langle {\gamma_L}', e_3 \rangle \neq 0
\quad \text{and} \quad
{\gamma_L}' \nparallel e_3
\quad \text{at $z$},
\]
\item[(3)]
$f$ is $\mathcal{A}$-equivalent to the swallowtail at $z$ if and only if
\[
\langle {\gamma_L}', e_3 \rangle \neq 0,
\quad
{\gamma_L}' \parallel e_3,
\quad \text{and} \quad
{\gamma_L}'' \nparallel e_3,
\quad \text{at $z$},
\]
\item[(4)]
$f$ is $\mathcal{A}$-equivalent to the cuspidal cross cap at $z$ if and only if
\[
\langle {\gamma_L}', e_3 \rangle = 0,
\quad
{\gamma_L}'' \nparallel e_3,
\quad \text{and} \quad
\langle {\gamma_L}'', e_3 \rangle \neq 0,
\quad \text{at $z$}.
\]
\end{itemize}
\end{theorem}

We can see in \cite{BK} assertions similar to Theorem \ref{thm: criteria2} for the case of maximal surfaces in $(\Nil,g_-)$.

\begin{proof}
It follows from \eqref{eq: TCMC representation} and \eqref{eq: gamma'} that
\begin{align}
{\gamma_L}'
=
df_L({\gamma}')
%&=
%\frac{1}{\tau} \Re \left[ j \hat\omega \left( \overline{\left({\frac{g_z}{g}}\right)} +\frac{g_{\bar{z}}}{g} \right) \begin{pmatrix} g^2 +1\\ j(g^2-1)\\ 2jg \end{pmatrix} \right]\\
&=
\frac{2}{\tau} \Im \left[ \hat\omega \left( \overline{\left({\frac{g_z}{g}}\right)} +\frac{g_{\bar{z}}}{g} \right) \begin{pmatrix} g^2 +1\\ j(g^2-1)\\ 2jg \end{pmatrix} \right].
\end{align}
Based on the relations $|g(\gamma(t))|^2=1$ and $g_{\bar{z}}(\gamma(t)) = 4j \overline{\hat\omega (\gamma(t))}$, each component of the above equation can be computed as follows.
\begin{align}
 \Im \left[ \hat\omega \left( \overline{\left({\frac{g_z}{g}}\right)} +\frac{g_{\bar{z}}}{g} \right) (g^2 +1) \right]
&=
 \Im \left[ \hat\omega \overline{\left({\frac{g_z}{g}}\right)} (g^2 +1) + 4j|\hat\omega|^2 \left( g+\frac{1}{g} \right) \right]\\
&=
 \Im \left[ |\hat\omega|^2 \overline{\left(\frac{g_z}{g^2\hat\omega}\right)} \left( g+\frac{1}{g} \right) +4j|\hat\omega|^2 \left( g+\frac{1}{g} \right) \right]\\
%&=
%\frac{1}{\tau}\Im\left[ 2|\hat\omega|^2 \Re(g) \left( \overline{\left(\frac{g_z}{g^2\hat\omega}\right)} + 4j \right) \right]\\
&=
2|\hat\omega|^2 \Re(g) \left( 4- \Im \frac{g_z}{g^2\hat\omega} \right),
\end{align}
\begin{align}
 \Im \left[ \hat\omega \left( \overline{\left({\frac{g_z}{g}}\right)} +\frac{g_{\bar{z}}}{g} \right) j(g^2-1) \right]
&=
 \Im \left[ j\hat\omega \overline{\left({\frac{g_z}{g}}\right)} (g^2-1) + 4|\hat\omega|^2 \left( g-\frac{1}{g} \right)\right]\\
&=
 \Im \left[ j|\hat\omega|^2 \overline{\left(\frac{g_z}{g^2\hat\omega}\right)} \left( g-\frac{1}{g} \right) + 4|\hat\omega|^2 \left( g-\frac{1}{g} \right) \right]\\
%&=
%\frac{1}{\tau} \Im \left[ 2|\hat\omega|^2 \Im(g) \left( \overline{\left(\frac{g_z}{g^2\hat\omega}\right)} + 4j \right) \right]\\
&=
2|\hat\omega|^2 \Im(g) \left( 4- \Im \frac{g_z}{g^2\hat\omega} \right),
\end{align}
\begin{align}
 \Im \left[ \hat\omega \left( \overline{\left({\frac{g_z}{g}}\right)} +\frac{g_{\bar{z}}}{g} \right) 2jg \right]
&=
\Im \left[ \hat\omega \overline{\left({\frac{g_z}{g}}\right)} 2jg + 8|\hat\omega|^2 \right]\\
&=
 \Im\left[ 2j |\hat\omega|^2 \overline{\left(\frac{g_z}{g^2\hat\omega}\right)} + 8|\hat\omega|^2 \right]\\
&=
2|\hat\omega|^2 \Re \frac{g_z}{g^2\hat\omega}.
\end{align}
Hence ${\gamma_L}'$ has the following representation:
\begin{equation}\label{eq: gamma_L'}
{\gamma_L}' = \frac{4|\hat\omega|^2}{\tau} \begin{pmatrix} \Re (g) \left( 4- \Im \dfrac{g_z}{g^2\hat\omega} \right)\\ \Im (g) \left( 4- \Im \dfrac{g_z}{g^2\hat\omega} \right)\\ \Re \dfrac{g_z}{g^2\hat\omega} \end{pmatrix}.
\end{equation}
By comparing each component of \eqref{eq: gamma_L'} with the criteria of Theorem \ref{thm: criteria} and its proof, we obtain the assertions of Theorem \ref{thm: criteria2}.
\end{proof}
\begin{remark}
The criteria in Theorem \ref{thm: criteria2} using the curve $\gamma_L=f_L \circ \gamma$ on the corresponding surface $f_L$ do not depend on the choice of parameter of the singular curve $\gamma$ and do not require the normal Gauss map $g$ of $f$. Therefore, such criteria are useful when the CMC surface in $\L$ corresponding to a timelike minimal surface in $\Nil(\tau)$ is known concretely.
\end{remark}

%%%%%%%%%%%%%%%%%%%%%%%%%%%%%%%%%%%
%%%%%%%%%%%%%%%%%%%%%%%%%%%%%%%%%%%
\section{Examples}\label{sec: example}
We recall a class of surfaces called {\it B-scrolls}, first introduced by Graves \cite{Graves}. These surfaces have properties with no counterpart in spacelike settings, and thus provide characteristic examples in timelike surface theory. The timelike minimal surfaces in $\Nil(\tau)$ corresponding to CMC $\tau$ B-scrolls are called {\it B-scroll type minimal surfaces} (see \cite{KK}). Like B-scrolls, they also exhibit features not found in spacelike geometry. The second named author explicitly constructed B-scroll type minimal surfaces and the corresponding B-scrolls, by elementary computations in \cite{K}.
Applying Theorem~\ref{thm: criteria2}, we obtain the following examples. For more details on B-scroll type minimal surfaces with singularities, we refer the reader to \cite{AK2}.

For a given constant $\tau$, let $B$ be an $\R^3$-valued function defined on an open interval $I$ which is a solution of the Frenet-Serret equations
\begin{equation}\label{eq: B-scroll}
(A\ B\ C)' = (A\ B\ C) \begin{pmatrix} 0&0&\tau\\0&0&\kappa\\\kappa&\tau&0 \end{pmatrix},
\end{equation}
where $\kappa$ is a smooth function on $I$ and $(A,B,C)$ is a null frame satisfying
\begin{equation}\label{eq:frame_cond}
\langle A,A\rangle=\langle B,B\rangle=0,\quad \langle A,B\rangle=-1,\quad C=A\times B.
\end{equation}

Let us define a map $f_L\colon I \times \R \to \R^3$ by
\[f_L(s,t) = \int A(s) ds + t B(s).\]
Then, $f_L$ becomes a timelike CMC $\tau$ surface in $\L$, which is is called a {\it B-scroll}.
%%%%%%%%%%%%%%%%%%%%%%%%%%%%%%%%%%
\begin{example}[{\cite[Example $6.3$]{AK2}}]
It can be easily checked that a matrix-valued function $(A(s)\ B(s)\ C(s))$ defined by
\[
A(s)=\begin{pmatrix} \cosh 2s\\ 1\\ -\sinh 2s \end{pmatrix},
\quad
B(s)=\frac12\begin{pmatrix} \cosh 2s\\ -1\\ -\sinh 2s \end{pmatrix},
\quad
C(s)=\begin{pmatrix} \sinh 2s\\ 0\\ -\cosh 2s\end{pmatrix}
\]
is a solution of the B-scroll condition \eqref{eq: B-scroll} with $\kappa=2$ and $\tau=1$.
Then the provided B-scroll $f_L$ is represented as
\[
f_L(s,t) = \frac12 \begin{pmatrix} \sinh 2s + t\cosh 2s\\ 2s-t\\ -1 -\cosh 2s -t\sinh 2s \end{pmatrix}.
\]
Consequently, Theorem \ref{thm: criteria2} shows that the corresponding B-scroll type minimal surface has the swallowtails at $(s,t) = \left( s^{\pm}, -\dfrac{C_3(s^{\pm})}{B_3(s^{\pm})} \right)$ where $s^{\pm} = \frac14 \log \left(5 \pm 2\sqrt{6}\right)$, $B_3(s) = -\frac12 \sinh 2s$, and $C_3(s) = -\cosh 2s$ (see Figure \ref{Fig:SW}).
\begin{figure}[h!]
\vspace{-0.5cm}
\begin{center}
 \begin{tabular}{{c@{\hspace{-10mm}}c@{\hspace{-10mm}}c}}
\hspace{-5mm}   \resizebox{6.0cm}{!}{\includegraphics[clip,scale=0.30,bb=0 0 555 449]{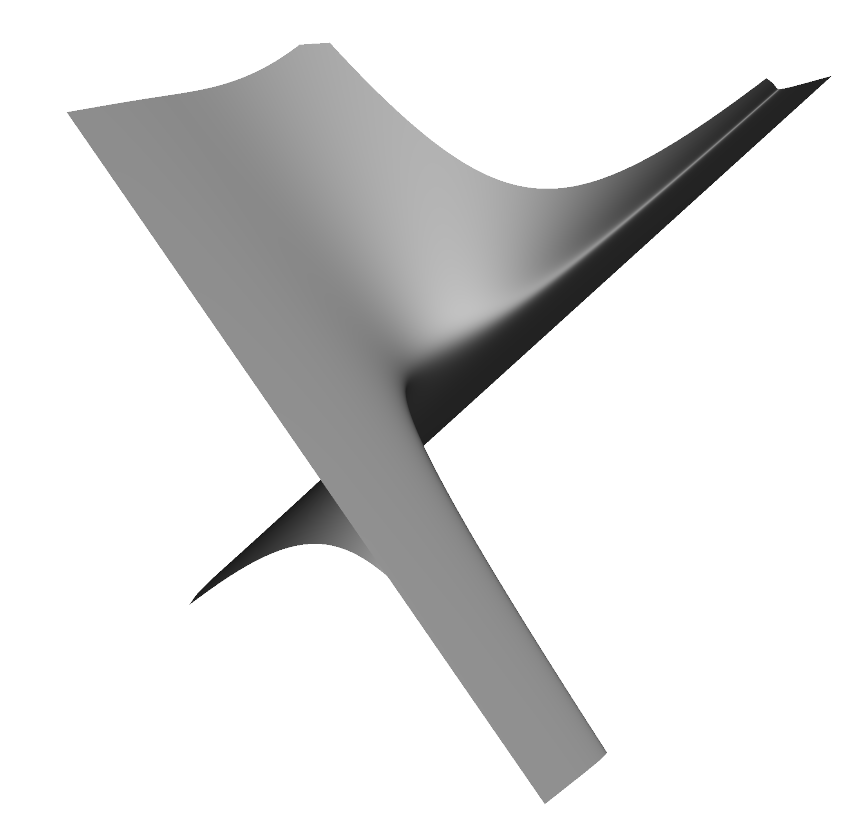}}&
 \hspace{-0mm} 
\resizebox{6.5cm}{!}{\includegraphics[clip,scale=0.30,bb=0 0 555 449]{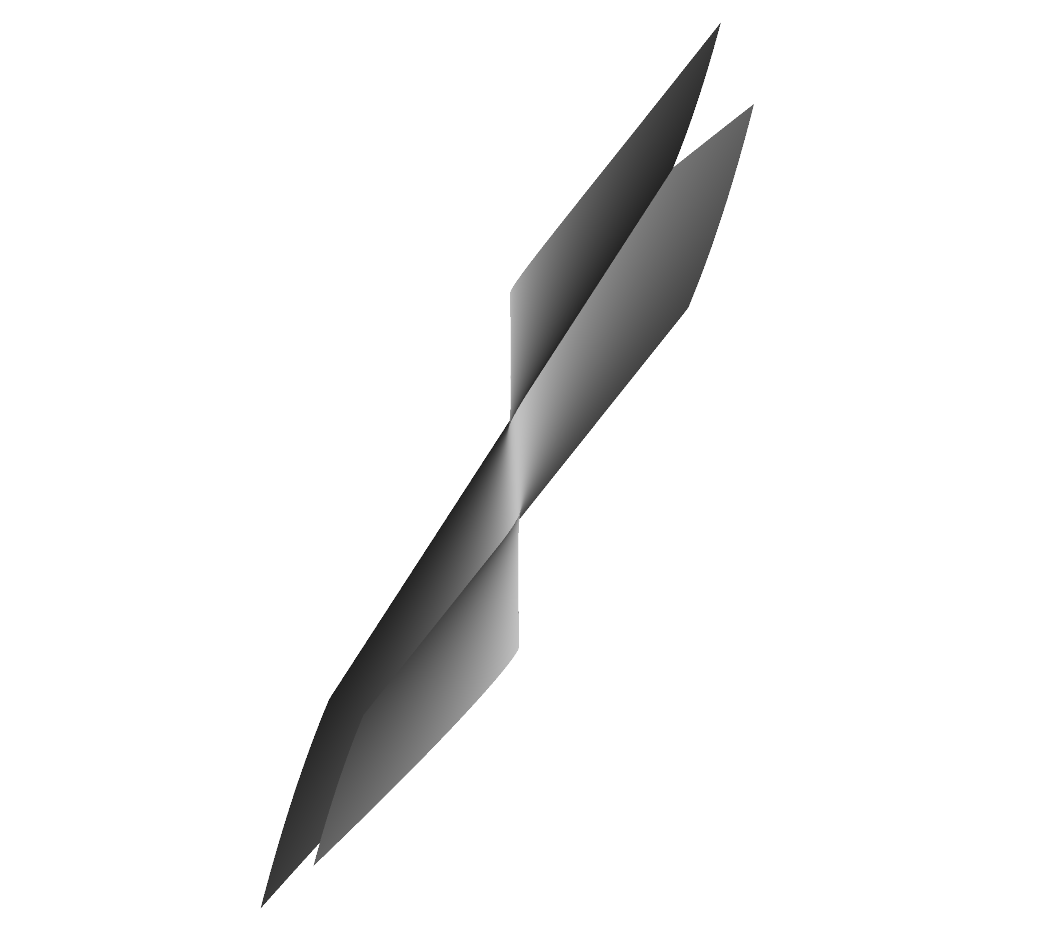}} &
 \hspace{-0mm} 
\resizebox{5.5cm}{!}{\includegraphics[clip,scale=0.30,bb=0 0 555 449]{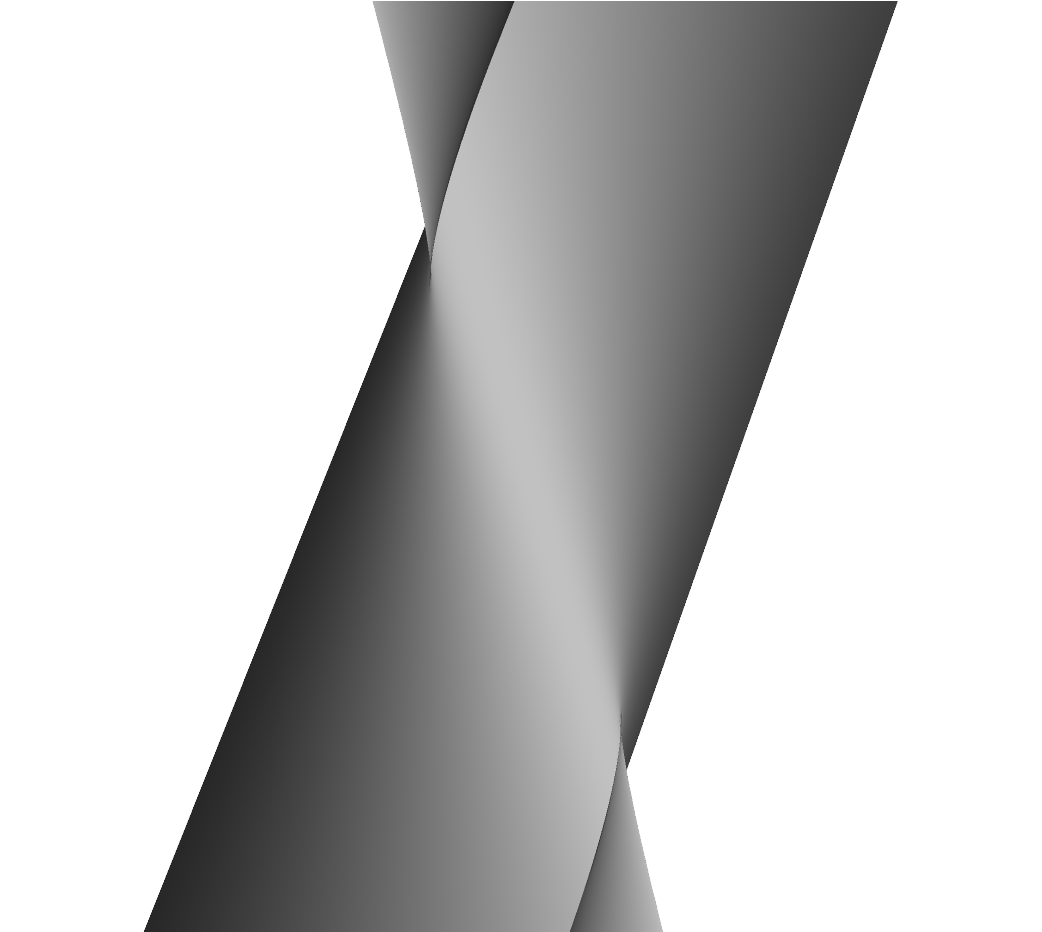}} \\
  {\hspace{-15mm}\footnotesize  $f_L$ in $\mathbb{L}^3$} &
  {\hspace{-15mm}\footnotesize  $f$ in $\Nil(1)$}&
  {\hspace{-15mm}\footnotesize  two swallowtails of $f$}
 \end{tabular}
 \caption{A timelike constant mean curvature (CMC) $1$ B-scroll $f_L$ in $\mathbb{L}^3$ (left) and the corresponding timelike minimal B-scroll type surface $f$ in $\Nil(1)$ with swallowtails (center and right). }
 \label{Fig:SW}
\end{center}
\end{figure}
\end{example}

\begin{example}[{\cite[Example $6.1$]{AK2}}]
We do not go into the details here, but a matrix-valued function $(A\ B\ C)$ satisfying the assertion $(4)$ in Theorem \ref{thm: criteria2} at some point can be explicitly constructed. For instance, $(A\ B\ C)$ derived from the B-scroll condition \eqref{eq: B-scroll} with $\kappa=-(6 - 36 s^2)/(1 + 3 s^2)^2$ and $\tau=1$ induces a B-scroll type minimal surface with cuspidal cross caps at $\left(s^{\pm}, -C_3(s^{\pm})/B_3(s^{\pm}) \right)$ where $s^{\pm} = \pm1/\sqrt{6}$. Here, $B_3(s)$ and $C_3(s)$ denote the third component of $B(s)$ and $C(s)$ (see Figure \ref{Fig:CCR}).
\begin{figure}[h!]
\vspace{-0.5cm}
\begin{center}
 \begin{tabular}{{c@{\hspace{-10mm}}c}}
\hspace{10mm}   \resizebox{8.0cm}{!}{\includegraphics[clip,scale=0.30,bb=0 0 555 449]{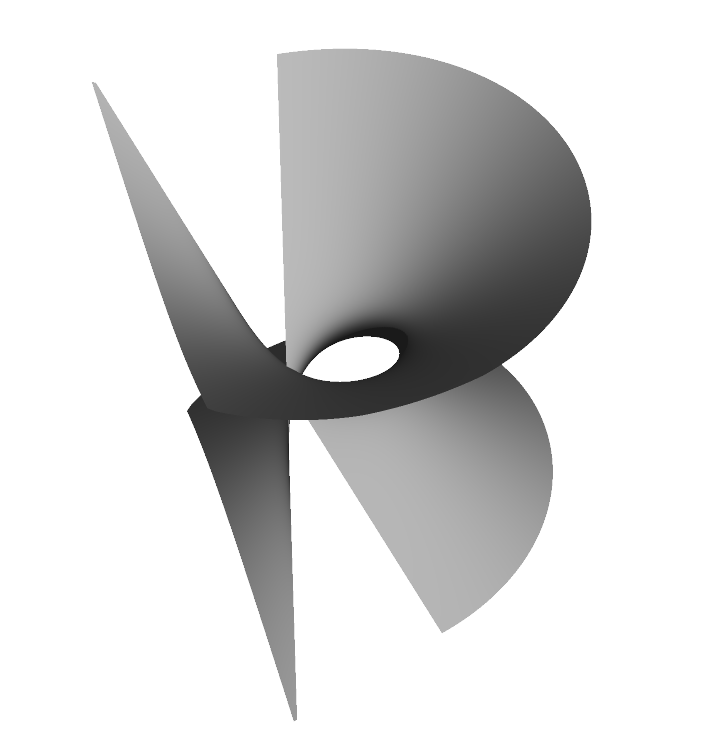}}&
 \hspace{-0mm} 
\resizebox{8.5cm}{!}{\includegraphics[clip,scale=0.30,bb=0 0 555 449]{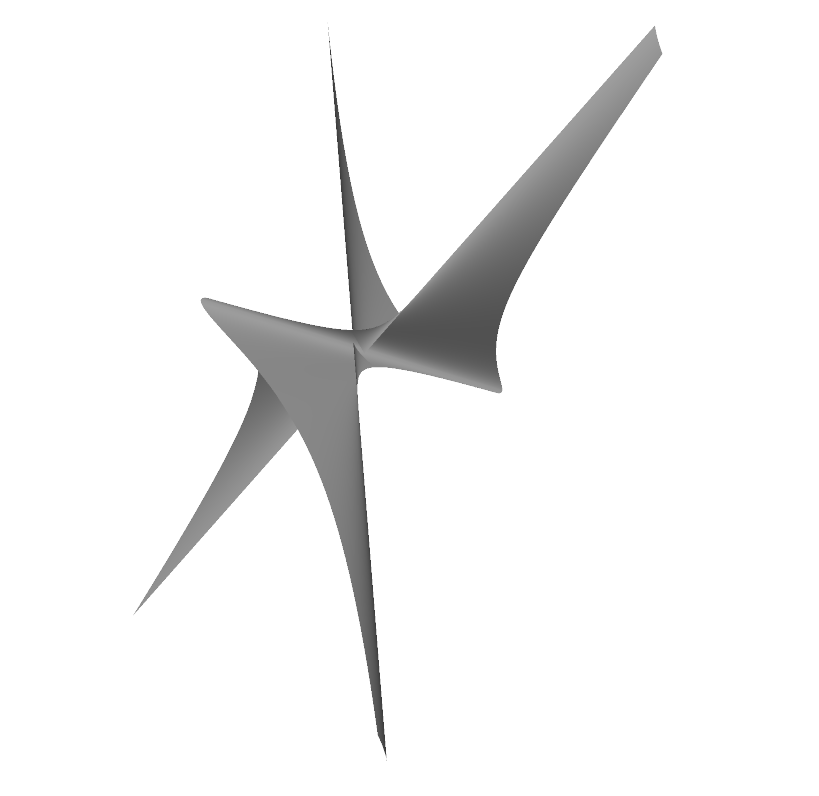}} \\
  {\hspace{-15mm}\footnotesize  $f_L$ in $\mathbb{L}^3$} &
   {\hspace{-35mm}\footnotesize  $f$ in $\Nil(1)$}
 \end{tabular}
 \caption{A timelike constant mean curvature (CMC) $1$ B-scroll $f_L$ in $\mathbb{L}^3$ (left) and the corresponding timelike minimal B-scroll type surface $f$ in $\Nil(1)$ with cuspidal cross caps (right). }
 \label{Fig:CCR}
\end{center}
\end{figure}
\end{example}

Accordingly, we have verified the following existence theorem.
\begin{theorem}\label{thm: existence}
There exist timelike minimal surfaces that have cuspidal edges, swallowtails, or cuspidal cross caps.
\end{theorem}
%%%%%%%%%%%%%%%%%%%%%%%%%%%%%%%%%%

\end{document}